\documentclass{article}
\usepackage[T2A]{fontenc}
\usepackage[cp1251]{inputenc}
\usepackage[russian,english]{babel}
\usepackage[tbtags]{amsmath}
\usepackage{amsfonts,amssymb}
\usepackage{graphicx}
\usepackage{amsthm}
\usepackage[left=2.5cm,right=2.5cm,top=2.5cm,bottom=2.5cm]{geometry}

\theoremstyle{plain}

\newtheorem{theorem}{Theorem}

\newtheorem{lemma}{Lemma}
\newtheorem{corollary}{Corollary}

\newtheorem{remark}{Remark}

\begin{document}

\begin{center} {Boundary value problems for some quasilinear parabolic equations \\ under minimal conditions on the coefficients}

\vspace{0.5cm}

S.G.  Pyatkov

\vspace{0.5cm}

Engineering School  of Digital Technologies, Yugra State University,  \\
  Khanty-Mansiysk, Russia, s\_pyatkov@ugrasu.ru

\noindent\parbox{11cm}{ \abstract{Quasilinear parabolic boundary value problems  are considered. 
 Sharp conditions  on the data ensuring existence and uniqueness  of solutions in Sobolev
classes  are exposed. They are smoothness  and consistency  conditions on the
data and some inequalities ensuring the maximum  principle for a  solution.  The proof relies on
a priori bounds and the method of  continuation in a parameter. } }
\end{center}


\medskip


\section{Introduction}

 We examine quasilinear parabolic equations of the form 
\begin{equation}\label{e1}
	Mu= u_t-\sum_{i,j=1}^n \partial_{x_i}(a_{ij}(t,x,u)u_{x_j})+ b(x,t,u,\nabla u)=0,
\end{equation}
where $(t,x)\in Q = (0,T) \times G$,   $G\subset {\mathbb R}^n$ is a bounded domain in ${\mathbb R}^n$ with boundary   $\Gamma$. Let    $S= (0,T)\times\Gamma$.
The initial and boundary conditions are written in the form either
\begin{equation}\label{e2}
	u|_{t=0}=u_0,\ \ u|_{S}= g(t,x), 
\end{equation}
or
\begin{equation}\label{e3}
	u|_{t=0}=u_0,\ \ \frac{\partial u}{\partial N}+ \psi(t,x,u)|_{S}=0, \ \frac{\partial u}{\partial N}= \sum_{i,j=1}^na_{ij}u_{x_j}\nu_i.
\end{equation}

Equations of parabolic type are encountered in many areas of mathematics and mathematical physics, and those encountered most frequently are linear and quasi-linear parabolic equations of the second order. They  are considered by many authors. First, we should refer  to the monographs \cite{lad}, \cite{lie}, where solutions to the problems 
\eqref{e1}, \eqref{e2} and  \eqref{e1}, \eqref{e3} are sought   in the H\"{o}lder spaces. In our opinion,  the conditions on the coefficients of the equation here are far from optimal.
The problem \eqref{e1}, \eqref{e2} in Sobolev spaces is  considered in \cite[Theorem 4, Ch. 15, Sect. 1]{kry1}, where the functions $a_{ij}$ meet the Lipschitz conditions in $u$ and the function 
  $b(t,x,u,p)$ has at most linear growth in the variables $(u,p) $ and satisfies the Lipschitz condition in $p$.
The local (in time) solvability results in \cite{wei} for the problem  \eqref{e1}, \eqref{e3} are proven under the condition that the functions $a_{ij}$ are of the class $C^3$ in all variables
and the function $b$ meets the Lischitz condition in $(u,p)$. 
Much weaker conditions on the functions $a_{ij}$ with  almost the same conditions on $b$ are used in the article \cite{res} (see also the references therein), where the functions $a_{ij}$ are only $C^1$-functions in $u$ and the Dirichlet boundary conditions were used.
The semigroup approach to quasilinear evolution problems  has been developed  in \cite{lun}, \cite{yagi}, \cite{ama1}. Here in the quasilinear case the functions spaces used consist of functions H\"{o}lder continuous or continuously differentiable with respect to $t$. The weighted Sobolev spaces 
are used in \cite{cro}, \cite{koh} also in the abstract setting, where the approach based on maximal $L_p$-regularity is involved. In this case, as well as in many other 
articles devoted abstract problems,   the main part of the operator is independent of $t$. 
Properties of viscosity solutions are studied in \cite{silv} and many other articles (see \cite{kry1}). There are a lot of articles devoted to different model cases (see, for instance, \cite{maz1}).
  
Our main results are connected with existence and uniqueness theorems  for solutions to the problems  \eqref{e1}, \eqref{e2} and  \eqref{e1}, \eqref{e3}
under the minimal smoothness conditions on the data of the problems, in particular, on the coefficients of  the parabolic operator and the function $\psi$ in \eqref{e3}.
A solution is sought in the class $W_p^{1,2}(Q)$. To establish the H\"{o}lder continuity of a solution,  we use the approach proposed in \cite{lad}.
In contrast to the results of \cite{lad}, \cite{lie}, to prove   maximum estimates for a solution, we use a different  approach, with allows 
easily prove  maximum estimates in the case of the boundary condition \eqref{e3} for generalized solutions which are absent in  \cite{lad}, \cite{lie}.
The $W_p^{1,2}(Q)$- estimates of a solution relies on arguments with the use of frozen coefficients. The results can be used in: establishing solvability results  for a large class of inverse parabolic problems; establishing existence and uniqueness of solutions under perturbations;  providing constructive approximation methods via iterative techniques.

\section{Preliminaries}

Let $E$ be a Banach space.
The notations for Sobolev and Besov spaces $W_p^s(G;E)$,  $B_{pq}^s(G;E)$, $W_p^s(Q;E)$, etc.
    are conventional (see \cite{tri,ama}). If $E={\mathbb R}$ or
    $E={\mathbb R}^n$ then we denote the last space simply
    by  $W_{p}^{s}(Q)$. The  H\"{o}lder spaces
    $C^{\alpha,\beta}(\overline{Q}),
    C^{\alpha,\beta}(\overline{S})$ are defined, for example,  in
    \cite{lad}.
 Given an interval $J=(0,T)$, denote
    $W_p^{s,r}(Q)=W_p^{s}(J;L_p(G))\cap L_p(J;W_p^r(G)) $.
    Similarly,   $W_p^{s,r}(S)=W_p^{s}(J;L_p(\Gamma))\cap
    L_p(J;W_p^r(\Gamma))$.
Let     $(u,v)=\int_{G} u(x)v(x)\, dx$. All considered spaces and coefficients of the equation \eqref{e1} are assumed to be
      real.            Denote    $Q_{\tau}=(0,\tau)\times G$ and  $S_\tau=(0,\tau)\times \Gamma$.
We assume that  $\Gamma\in C^2$,   (see the definition  in \cite[Chapter 1]{ lad}).

Proceed with one auxiliary result. 
Consider the linear problem 
\begin{equation}\label{a1}
	Mu= u_t-\sum_{i,j=1}^n \partial_{x_i}(a_{ij}u_{x_j})+ \sum_{i=1}^n a_iu_{x_i}+ a_0u=f,
\end{equation}
where $(t,x)\in Q = (0,T) \times G$. Initial and boundary conditions are written as follows: 
either
\begin{equation}\label{a2}
	u|_{t=0}=u_0,\ \ u|_{S}= g(t,x), 
\end{equation}
or
\begin{equation}\label{a3}
	u|_{t=0}=u_0,\ \ \frac{\partial u}{\partial N}+\sigma u|_S=g, \ \frac{\partial u}{\partial N}= \sum_{i,j=1}^na_{ij}u_{x_j}\nu_i,
\end{equation}
with $\vec{\nu}=(\nu_1,\ldots,\nu_n)$ the outward unit normal to $\Gamma$. 
It is assumed that 
\begin{equation}\label{a4}
u_0\in W_p^{2-2/p}(G),\  g\in W_p^{k_0,2k_0}(S), \  f\in L_p(Q), \  \Gamma\in C^2,\ p>\frac{n+2}{2},
\end{equation}
where $k_0=s_1=1-1/2p$ in the case of the  conditions   \eqref{a2} and $k_0=s_0=1/2-1/2p$ otherwise.  Other conditions on the data are as follows: 
there exists a constant $\delta_0>0$ such that
\begin{equation}\label{a5}
\delta_0|\xi|^2\leq  \sum_{i,j=1}^n a_{ij}\xi_j\xi_i\leq |\xi|^2/\delta_0\  \forall \xi\in {\mathbb R}^n;
\end{equation}
\begin{equation}\label{a501}
a_{ij}\in C(\overline{Q}), \ a_{ijx_i}, a_i\in L_{q_2}(Q) \ (q_2>n+2, q_2\geq p),\ a_0\in L_{q_1}(Q) \ (q_1>\frac{n+2}{2},\ q_1\geq p);
\end{equation}
\begin{equation}\label{a502}
a_{ij},\sigma \in B_{q_3p}^{s_0}(0,T; L_{q_3}(\Gamma))\cap L_{q_3}(0,T; B_{q_3p}^{2s_0}(\Gamma)), \ \frac{n+1}{q_3}< 1-\frac{1}{p}=2s_0, \ q_3\geq p,
\end{equation}
where $i,j=1,2,\ldots,n$.

The following theorem is valid   (we use slightly stronger conditions than those in
 \cite[Theorem 3.4]{lind} or \cite[Theorem 5.3]{lind1}, see also  \cite{den},    \cite[Theorem 2.1]{pya11}). 
   
\begin{theorem} \label{thr0}   Let  $ p>(n+2)/2$ and let the conditions 
{\rm     \eqref{a4}-\eqref{a501}} in the case of the problem {\rm \eqref{a1}, \eqref{a2}} or the conditions    {\rm  \eqref{a4}-\eqref{a502}} in the case of the problem  {\rm \eqref{a1}, \eqref{a3}} hold. In the latter case we also assume that $p\neq 3$.
 Then there exists a unique solution   $u$ to the problems  {\rm \eqref{a1}, \eqref{a2}} and {\rm \eqref{a1}, \eqref{a3}}  such that 
 $u\in W_p^{1,2}(Q)$ and $u$ satisfies the estimate 
\begin{equation}\label{a6}
 \|u\|_{W_{p}^{1,2}(Q)}\leq
C_{0}(\|u_{0}\|_{W_{p}^{2-2/p}(G)}+
\|f\|_{L_{p}(Q)}+\|g\|_{W_{p}^{k_{0},2k_{0}}(S)}),
\end{equation}
where we can assume that the constant  $C_0$ depends only on  $\delta_0$. It depends also on the norms of coefficients  and the domain $G$ but
it is bounded whenever these  norms in the spaces indicated  in \eqref{a501}, \eqref{a502} are bounded. 
\end{theorem} 

Let $v^+(t,x)=\max(v(t,x),0)$.
The following theorem results from Lemma 1.2.3 in \cite{maz}. 

\begin{theorem}\label{thr1} Let $u$ be a nonnegative measurable function in $Q$. Then 
\begin{multline*}
\int_Q u\,dxdt= \int_0^\infty \mu(M_k)dk=\int_0^\infty \mu(L_k)dk, \\ M_k=\{(t,x)\in Q: u(t,x)\geq k\}, L_k=\{(t,x)\in Q: u(t,x)>  k\},
\end{multline*} 
where $\mu(\cdot)$ is the Lebesgue measure.
\end{theorem}
\begin{corollary}\label{col1} Let  $u$ be a nonnegative measurable function in  $Q$. Then 
$$
\int_Q (u-k)^+\,dxdt= \int_k^\infty \mu(M_\tau)\,d\tau =\int_k^\infty \mu(L_\tau)\,d\tau, \ k>0.\
$$ 
\end{corollary}\label{col1}
\begin{proof} Indeed, 
$$
\int_Q(u-k)^+\,dxdt= \int_0^\infty \mu(\tilde{M}_\xi)d\xi =\int_k^\infty \mu(\tilde{L}_\xi) d\xi,\
$$
where $\tilde{M}_\xi=\{(t,x):\ (u-k)^+\geq \xi\}=\{(t,x):\ u(t,x)\geq k+\xi\}$. In this case, we infer 
$\tilde{M}_\xi=M_{k+\xi}$ and
 $\int_Q (u-k)^+\,dxdt= \int_0^{\infty}\mu(M_{k+\xi})\,d\xi$. Make the change of variables  $k+\xi=\tau$. We obtain that 
$\int_Q (u-k)^+\,dxdt=\int_k^\infty \mu(M_\tau)\,d\tau$. The case of the set $L_\xi$ is considered by analogy.
\end{proof}

Below, we present an  inequality similar to that in Sect. 3 of Ch. 2 in \cite{lad}.

\begin{lemma}\label{lem1}
Let  $u\in L_\infty(0,T;L_2(Q))\cap L_2(0,T;W_2^1(G))$. Then  $u\in L_r(Q)$, $r\in [2, 2(n+2)/n]$, and 
\begin{multline*}
\|u\|_{L_r(Q)}\leq c_0\|u\|_{L_\infty(0,T;L_2(G))}^{1-s} \|u\|_{L_2(0,T;W_2^1(G))}^s\leq \\
c_0(\|u\|_{L_\infty(0,T;L_2(G))}^2+ \|u\|_{L_2(0,T;W_2^1(G))}^2)^{1/2}=c_0|u|_1, 
\end{multline*}
where $\  s=n(r-2)/2r.$
\end{lemma}
\begin{proof}
The inequality follows from the interpolation inequality    $\|u\|_{W_p^s(G)}\leq c \|u\|_{W_p^1(G)}^s \|u\|_{L_p(G)}^{1-s}$ \cite{tri} 
and the embedding  $W_p^s(G)\subset L_r(G)$ ($r=2n/(n-2s)$) \cite[Sect.  4.6.1, 4.6.2]{tri}. It suffices to prove the inequality 
for $r=2(n+2)/n$, in this case 
$r=2n/(n-2s)$ with $s=n/(n+2)$ and 
\begin{multline*}
\|u\|_{L_r(Q)}\leq \Bigl(\int_0^T \|u\|_{L_r(G)}^r\,dt\Bigr)^{1/r}\leq \Bigl(\int_0^T c_1\|u\|_{W_2^s(G)}^r\,dt\Bigr)^{1/r}\leq \\
(c_1c)^{1/r}\Bigl(\int_0^T  \|u\|_{W_p^1(G)}^{2} \|u\|_{L_2(G)}^{r(1-s)}\,dt\Bigr)^{1/r}\leq (c_1c)^{1/r} 
\|u\|_{L_2(0,T;W_2^1(G))}^s \|u\|_{L_\infty(0,T;L_2(G))}^{1-s}.
\end{multline*}
\end{proof}

\section{Basic Results}

Denote  $\Gamma_T=S\cup \{(0,x): x\in G\}$.
The functions $a_{ij}$ are assumed to be continuous in $\overline{Q}\times (-\infty,\infty)$ and
\begin{equation}\label{e6}
\sum_{i,j=1}^na_{ij}p_ip_j\geq \delta_0|p|^2\  \forall \vec{p}=(p_1,\ldots,p_n) \in {\mathbb R}^n, \ (t,x)\in Q, \ u\in {\mathbb R},
\end{equation}
where $\delta_0>0$ is a constant, the function 
  $b(t,x,u,\vec{p})$ meets the Caratheodory condition and
\begin{equation}\label{e7}
b(t,x,u, \vec{p})u\geq -\beta_0a_{ij}p_ip_j -  g_2 |u|^2-g_1|u|,\  \forall \vec{p}\in {\mathbb R}^n, \ \textrm{a.e.}\ (t,x)\in Q, \  |u|\geq \tilde{m}_0, \ p\in {\mathbb R}^n.
\end{equation}
for some constant  $\tilde{m}_0>0$. Here the abbreviation a.e. stands for almost  everywhere.
Moreover, for every $R>0$, there exists a constant  $\beta_1>0$ and a function $g_3$ such that 
\begin{equation}\label{e8}
|b(t,x,u, \vec{p})|\leq \beta_1 |\vec{p}|^2 +g_3(t,x)\   \forall  (t,x)\in Q, \  u\in [-R,R], \ \vec{p}\in {\mathbb R}^n,
\end{equation}
where  $g_i\in L_{p}(Q)$, $i=1,2,3$, $g_1,g_2,g_3\geq 0$ a.e., $\beta_i$ are nonnegative constants. 

We look for a solution to the problems   $\eqref{e1}, \eqref{e2}$ and $\eqref{e1}, \eqref{e3}$ in the space 
$W_p^{1,2}(Q)$.
We additionally assume that for every  $R>0$
\begin{equation}\label{e2051}
a_{ijx_i}\in L_{q_3}(Q;C([-R,R])), \ a_{iju}\in L_\infty(Q;C([-R,R])),   \ q_3> n+2,\ q_3\geq p,\ i,j=1,2,\ldots,n. 
\end{equation}
In the case of the problem  \eqref{e1}, \eqref{e3}, we require that 
\begin{equation}\label{e2052}
\psi\in W_p^{s_0,2s_0}(S;C([-R,R])),   \ \psi_u\in L_{q_4}(S;C([-R,R])), \   p>\frac{n+2}{2}, \ q_4> n+1, \ q_4\geq p,
\end{equation}
\begin{equation}\label{e2053}
a_{ij}\in B_{q_5p}^{s_0}(0,T; L_{q_5}(\Gamma;C([-R,R])))\cap L_{q_5}(0,T; B_{q_5p}^{2s_0}(\Gamma;C([-R,R]))) \ \forall R>0,\   q_5\geq p,
\end{equation}
where  $(n+1)/q_5< 1-1/p=2s_0$ and the derivatives are  just generalized derivatives in the Sobolev sense.

By a solution to the problems $\eqref{e1}, \eqref{e2}$ and $\eqref{e1}, \eqref{e3}$, we mean a function $u\in W_p^{1,2}(Q)$, $p>(n+2)/2$ satisfying the equation \eqref{e1} and the boundary condition \eqref{e2}, respectively, \eqref{e3} a.e. in $Q$, respectively, on $S$. 
Note that if $u\in W_p^{1,2}(Q)$ $(p>(n+2)/2)$ then the embedding theorems  ensure that 
 $u|_{t=0}=u_0\in W_p^{2-2/p}(G)\subset C^{2\varepsilon_0}(\overline{G})$, $\varepsilon_0=1-(n+2)/2p$,  
$u|_{S}=g\in C^{\varepsilon_0, 2\varepsilon_0}(\overline{S})$ (\cite[Theorems  1.19, 1.22]{pya11}).

\begin{theorem} \label{thr2}
Assume that  $u\in W_p^{1,2}(Q)$ $(p>(n+2)/2)$ is a  solution to the problem  {\rm \eqref{e1}-\eqref{e2}}, the conditions 
   {\rm \eqref{e6}-\eqref{e8}} hold, and   $g(t,x)\in L_\infty(S)$, $u_0\in L_\infty(G)$. 
Then there exists a constant  $M>0$, depending on $\delta_0, \beta_0$,  $\|g_2\|_{L_{p}(Q)}$, $\|u\|_{L_\infty(\Gamma_T)}$
such that  $\|u\|_{L_\infty(Q)}\leq M$. 
    \end{theorem}
      
\begin{proof} The idea of the proof is similar to that in Sect. 2 of Ch. 5 in \cite{lad} but we use Theorem 1 rather than Theorem 6.1 of Ch. 2 in \cite{lad}.   
Make the change of variables 
$u=e^{\lambda t}v$ ($\lambda >0$). The function  $v$ is a solution to the equation
\begin{equation}\label{e9}
	Mu= v_t-\sum_{i,j=1}^n \partial_{x_i}(a_{ij}(t,x,ve^{\lambda t})v_{x_j})+ e^{-\lambda t}b(x,t,ve^{\lambda t},e^{\lambda t}\nabla v)+\lambda v=0.
\end{equation}
The initial and boundary conditions are written as 
\begin{equation}\label{e10}
	v|_{t=0}=u_0,\ \ v|_{S}= g(t,x)e^{-\lambda t}, 
\end{equation}
Let  $w=|v|^m$, $w^{(k)}=(|v|^m-k)^+$, 
 $A_k(t)=\{x\in G:\  w(t,x)> k\}$. In this case  $\int_0^T A_k(t)\,dt=\mu(L_k)$, $L_k=\{(t,x): w(t,x)> k\}$.
Multiply  \eqref{e9} by   $|v|^{m-2} v w^{(k)}$ and integrate over  $G$. Choose  $m_0> \max(\|u_0\|_{L_\infty(G)},\|ge^{-\lambda t}\|_{L_\infty(S)}, \tilde{m}_0)$. 
In view of this choice,   $w^{(k)}|_{\Gamma_T}=0$ for $k\geq k_0=m_0^m$. 
Integrating by parts, we obtain that 
\begin{multline}\label{e11}
\frac{1}{m}\frac{\partial}{\partial t} \int_G (w^{(k)})^2\,dx +\int_{A_k(t)} 
\sum_{i,j=1}^n a_{ij}v_{x_j}v_{x_i} [(m-1)|v|^{m-2}w^{(k)}+ m|v|^{2m-2}] \\ + 
e^{-\lambda t}b(t,x,ve^{\lambda t},e^{\lambda t}\nabla v)|v|^{m-2} v w^{(k)}\,dx=0.
\end{multline}
The condition  \eqref{e7} yields 
\begin{multline*}
\frac{1}{m}\frac{\partial}{\partial t} \int_G (w^{(k)})^2\,dx +
\int_{A_k(t)} \sum_{i,j=1}^m a_{ij}v_{x_j}v_{x_i} [(m-1)|v|^{m-2}w^{(k)}+ m|v|^{2m-2}]+\lambda ww^{(k)}\,dx \leq  \\
\int_{A_k(t)} \beta_0a_{ij}v_{x_i}v_{x_j}|v|^{m-2}w^{(k)} + 
f|v|^{m}w^{(k)} \,dx, \  f=|g_2|+\frac{1}{m_0}g_1.
\end{multline*}
This inequality can be rewritten in the form 
\begin{multline*}
\frac{1}{m}\frac{\partial}{\partial t} \int_G (w^{(k)})^2\,dx +
\int_{A_k(t)} \sum_{i,j=1}^m a_{ij}v_{x_j}v_{x_i} [|v|^{m-2}w^{(k)}(m-1-\beta_0) + m|v|^{2m-2}]+\lambda ww^{(k)}\,dx \\ \leq 
\int_G f|v|^{m}w^{(k)} \,dx,\   
\end{multline*}
Next, we take  $m=\beta_0+1$.
In this case we arrive at the inequality 
\begin{equation}\label{e14}
\frac{1}{m}\frac{\partial}{\partial t} \int_{G} (w^{(k)})^2\,dx +
\frac{\delta_0}{2m}  \int_{G} |\nabla w^{(k)}|^2+ \lambda w w^{(k)}\,dx \leq  \int_G f w w^{(k)} \,dx. 
\end{equation}
Integrating in $t$, we conclude that 
\begin{equation}\label{e15}
\max_{t\in (0,T)} \int_{G} (w^{(k)})^2\,dx +
\int_{Q} |\nabla w^{(k)}|^2+ \lambda  w w^{(k)}\,dxdt \leq  c_0\int_{L_k} f w w^{(k)} \,dxdt, 
\end{equation}
where the constant  $c_0$ depends on  $m,\delta_0$. 
The H\"{o}lder inequality implies that 
$$
\int_{L_k} f w w^{(k)} \,dxdt \leq \|f\|_{L_{p}(Q)}\|ww^{(k)}\|_{L_{p'}(L_k)},\ 1/p+1/p'=1.
$$
Since $ww^{(k)}\leq 2 (w^{(k)})^2+k^2$, we have 
$\|ww^{(k)}\|_{L_{p'}(L_k)}\leq c_1( \|w^{(k)}\|_{L_{2p'}(G)}^2+ k^2(\mu(L_k))^{1/p'})$.
This inequality,    \eqref{e15}, and the inequality  $(w^{(k)})^2\leq  ww^{(k)}$ provide the estimate 
 \begin{equation}\label{e16}
\max_{t\in (0,T)} \int_{G} (w^{(k)})^2\,dx +
\int_{Q} |\nabla w^{(k)}|^2+ \lambda  (w^{(k)})^2\,dx \leq  c_2\|w^{(k)}\|_{L_{2p'}(Q)}^2+ c_2k^2 (\mu(L_k))^{1/p'}. 
\end{equation}
The Riesz-Thorin theorem or the results in  \cite[Sect. 1.18.4]{tri} ensure that 
$\|w^{(k)}\|_{L_{2p'}(Q)}\leq \|w^{(k)}\|_{L_{r}(Q)}^\theta \|w^{(k)}\|_{L_{2}(Q)}^{1-\theta}$,
where $r=2(n+2)/n$, $\theta r+(1-\theta)2=2p'$. Since   $p>(n+2)/2$,  $\theta\in (0,1)$.  
The inequality 
\begin{equation}\label{e64}
 |ab|\leq \varepsilon |a|^p/p+\varepsilon^{1-p} |b|^{p'}/p', \  p'=p/(p-1),\ \varepsilon>0, 
 \end{equation} 
 and lemma  \ref{lem1},  imply that 
\begin{equation*}
c_2\|w^{(k)}\|_{L_{2p'}(Q)}^2\leq  \frac{1}{2}|w^{(k)}|_{1}^2+ c_3\|w^{(k)}\|_{L_{2}(Q)}^2.
\end{equation*} 
In this case the inequality  \eqref{e16} yields 
 \begin{equation}\label{e17}
\max_{t\in (0,T)} \int_{G} (w^{(k)})^2\,dx +
\int_{Q} |\nabla w^{(k)}|^2+ 2\lambda  (w^{(k)})^2\,dx \leq  2c_2k^2(\mu(L_k))^{1/p'}+ 2c_3\|w^{(k)}\|_{L_{2}(Q)}^2. 
\end{equation}
Note that the constant $c_3$ is independent of $\lambda$.
Choosing  $\lambda >\max(1,2c_3)$, we conclude that 
 \begin{equation}\label{e18}
\max_{t\in (0,T)} \int_{G} (w^{(k)})^2\,dx +
\int_{Q} |\nabla w^{(k)}|^2+ \lambda  (w^{(k)})^2\,dx \leq  2c_2k^2 (\mu(L_k))^{1/p'}.
\end{equation}
Next, the H\"{o}lder inequality, Lemma \ref{lem1} and \eqref{e18} imply that 
\begin{multline*}
\int_Q w^{(k)}\,dQ \leq \|w^{(k)}\|_{L_{2(n+2)/n}(Q)}(\mu(L_k))^{(n+4)/2(n+2)}\leq c |w^{(k)}|_1 (\mu(L_k))^{(n+4)/2(n+2)}\\ \leq 
2c\sqrt{c_2}k(\mu(L_k))^{\alpha}, \  \alpha=(n+4)/2(n+2)+ 1/(2p')>1.
\end{multline*}
Thus, we obtain the inequality 
 \begin{equation}\label{e19}
\int_Q w^{(k)}\,dQ \leq c_3 k(\mu(L_k))^{\alpha}, \  \alpha=(n+4)/2(n+2)+ 1/(2p')>1.
\end{equation}
As a consequence, we  have
 \begin{equation}\label{e195}
\int_Q w^{(k_0)}\,dQ \leq c_3 k_0 (\mu(Q))^{\alpha}=M_0, \  \alpha=(n+4)/2(n+2)+ 1/(2p')>1.
\end{equation}
By Theorem  \ref{thr1} $\int_Q w^{(k)}=\int_k^\infty \mu(L_\tau)\,d\tau=f(k)$. In this case,  $f'(k)=-\mu(L_k)$ and the inequality  \eqref{e18} can be written in the form 
 $
f(k) \leq c_3 k((-f'(k))^{\alpha},
$
 or in the form  $f'(k)\leq - c_4 f(k)^{1/\alpha}/k^{1/\alpha}$, $c_4=c_3^{-1/\alpha}$. 
Integrating this inequality, we arrive at the relation  
$$
(f(k))^{1-1/\alpha}\leq (f(k_0))^{1-1/\alpha} - \frac{c_4 k^{1-1/\alpha}}{1-1/\alpha} +  \frac{c_4 k_0^{1-1/\alpha}}{1-1/\alpha}\leq c_5-  \frac{c_4 k^{1-1/\alpha}}{1-1/\alpha}
$$
which ensures that  $f(k)\equiv 0 $ for  $k>k_1=(\frac{c_5 (\alpha-1)}{c_4\alpha})^{\alpha/(\alpha-1)}$,  $c_5=(M_0)^{1-1/\alpha}+\frac{c_4 k_0^{1-1/\alpha}}{1-1/\alpha}$.  
Hence,  $|v|^m\leq k_1$ a.e.  
\end{proof}

Proceed with the problem  \eqref{e1},  \eqref{e3}. 
We assume that  $\psi\in W_p^{s_0,2s_0}(S;C([-R,R])$ for every  $R>0$ and
\begin{equation}\label{e201}
\psi(t,x,u)sgn\,u\geq -g_4-\beta_3 |u|, 
\end{equation}
where   $\beta_3\geq 0$  a constant and $g_4\in W_p^{s_0,2s_0}(S)$.  There exists a function $\rho\in C^2(\overline{G})$ such that 
 $\nabla \rho\cdot \nu=|\nabla \rho|$, $1/2\leq |\nabla \rho|\leq 2$ on $\Gamma$
(this function was constructed in the proof of Theorem  13.1 in \cite{lie}). 
Next, consider an even extension of the function $g_4$ for $t<0$ and construct $\varphi(t)=1$ for $t>-T/2$,  $\varphi(t)=0$ for $t<-3T/4$, $\varphi(t)\in C^1([-T,T])$. 
There exists a solution   $\rho_1(t,x)\in W_p^{1,2}((-3T/4,T)\times G)$ to the problem 
$\rho_{1t}-\Delta \rho_1=0,\ \frac{\partial \rho_1}{\partial \nu}|_S=g_4\varphi(t)$, $\rho_{1}(-3T/4,x)=0$ (Theorem \ref{thr0}). 
  
\begin{theorem} \label{thr3}
 Assume that  $u\in W_p^{1,2}(Q)$ $(p>(n+2)/2)$ is a solution to the problem  {\rm \eqref{e1}, \eqref{e3}}, the conditions 
   {\rm \eqref{e6}-\eqref{e8}, \eqref{e201}} hold, and    $u_0\in L_\infty(G)$. 
Then there exists a constant  $M>0$, depending on  $\beta_0$,  $\|g_2\|_{L_{p}(Q)}$,  $\|u_0\|_{L_\infty(G)}$, and
$\|g_4\|_{W_p^{s_0,2s_0}(S)}$ such that 
 $\|u\|_{L_\infty(Q)}\leq M$. 
    \end{theorem}
\begin{proof}
The proof is in line with that of Theorem \ref{thr2}. Repeating the arguments, we arrive at \eqref{e11} with an additional  summand $J= -\int_{\Gamma}e^{-\lambda t} \psi(t,x,e^{\lambda t}v) |v|^{m-2}v w^{(k)}\,d\Gamma$  on the right-hand side. Estimate it.   
In view of the conditions on  $\psi$, we have 
$$
J\leq \int_\Gamma (g_4e^{-\lambda t}+\beta_3|v|)|v|^{m-1}w^{(k)}\,d\Gamma.
$$
 Next, 
\begin{equation*}
\int_G \Delta\rho \beta_3|v| |v|^{m-1}w^{(k)}\,dx + \int_G \nabla \rho \cdot\nabla \beta_3|v|^{m}w^{(k)}\,dx= \int_{\Gamma}\nabla \rho\cdot \nu 
\beta_3|v|^{m}w^{(k)}\,d\Gamma. 
\end{equation*}
Therefore, we derive the inequality 
\begin{multline*}
 \int_{\Gamma} \beta_3|v|^{m}w^{(k)}\,d\Gamma \leq 
 2\int_{\Gamma}\nabla \rho\cdot \nu \beta_3|v|^{m}w^{(k)}\,d\Gamma\leq  \\
 2\int_G \Delta\rho \beta_3|v|^{m}w^{(k)}\,dx+ 2\int_G \nabla \rho \cdot\nabla (\beta_3|v|^{m}w^{(k)})\,dx\leq \\
c_1 \int_G \beta_3|v|^{m}w^{(k)}\,dx+ 2\int_G |\nabla (\beta_3|v|^{m}w^{(k)})|\,dx.
\end{multline*}
The first summand here is estimated by  the quantity
\begin{equation}\label{w21}
\int_G  \beta_3|v|^{m}w^{(k)}\,dx\leq  \beta_3\int_G w w^{(k)}\,dx\leq c_5\int_G (w^{(k)})^2\,dx+ c_5k^2 \mu(A_k(t)). 
\end{equation}
The estimate for the second summand  is of the form 
\begin{equation}\label{w22}
 \int_G |\nabla \beta_3|v|^{m}w^{(k)}|\,dx\leq  \varepsilon \int_{G} |\nabla w^{(k)}|^2  + c(\varepsilon) ((w^{(k)})^2\,dx +k^2 \mu(A_k(t))).
\end{equation}
Next, we use the equality 
\begin{equation*}
\int_G \Delta\rho_1  |v|^{m-2}v w^{(k)}\,dx+ \int_G \nabla \rho_1 \cdot\nabla (|v|^{m-2}v w^{(k)})\,dx= \int_{\Gamma} g_4 |v|^{m-2} v w^{(k)}\,d\Gamma 
\end{equation*}
which guarantees the relation
\begin{equation*}
 |\int_{\Gamma}e^{-\lambda t} g_4 |v|^{m-2} v w^{(k)}\,d\Gamma|  \leq 
 \int_{\Gamma}|\nabla \rho_1| |\nabla ( |v|^{m-2}v w^{(k)})|\,d\Gamma + \frac{1}{m_0}\int_G |\Delta\rho_1|  |v|^{m} w^{(k)}\,dx.
 \end{equation*}
The former  summand on the right-hand side of the previous inequality is bounded  by 
\begin{equation}\label{w24}
 \int_{G}|\nabla \rho_1| |\nabla |v|^{m-2}v w^{(k)}|\,dx\leq  \varepsilon \int_{G} |\nabla w^{(k)}|^2  + c(\varepsilon) (|\nabla \rho_1|^2 (w^{(k)})^2 + k^2\chi_{A_k(t)})\,dx.
\end{equation}
where $\chi_{A_k(t)}$ is the indicator of the set  $A_k(t)$.
Thus, the estimates   \eqref{w24} validate the relation
$$
\int_0^T \int_Ge^{-\lambda t} g_4 |v|^{m-2} v w^{(k)}\,dxdt \leq \varepsilon \int_{Q}|\nabla w^{(k)}|^2 + c(\varepsilon) \int_Q\tilde{g}_3( (w^{(k)})^2 +k^2\chi_{L_k}(t,x))\,dxdt,
$$
 where  $\tilde{g}_3=|\Delta \rho_1|+|\nabla \rho_1|^2$.
By the embedding theorems,  $\tilde{g}_3\in L_p(Q)$.  
The same arguments as those involved in the estimate of the right-hand side 
in \eqref{e15},  ensure the inequality
 $$
 \int_Q\tilde{g}_3 (w^{(k)})^2dxdt\leq \varepsilon_1 |w^{(k)}|_1^2  + c(\varepsilon_1) \int_Q |w^{(k)}|^2\,dxdt, 
$$
with $\varepsilon_1$ an arbitrary positive constant.
Moreover, 
 $$
 k^2 \int_Q \tilde{g}_3 \chi_{L_k(t,x)} \,dxdt\leq k^2\|\tilde{g}_3\|_{L_p(Q)}(\mu(L_k))^{1/p'},\  p'=p/(p-1). 
$$ 
These inequalities  for an appropriate  $\varepsilon_1$  and \eqref{w21}, \eqref{w22} validate the estimate 
$$
\int_0^T J\leq 3\varepsilon \int_{Q} |\nabla w^{(k)}|^2 + c(\varepsilon)( \int_Q |w^{(k)}|^2\,dxdt + k^2 (\mu(L_k))^{1/p'}).
$$
Thus, the inequality  \eqref{e18} from the previous theorem for an appropriate  $\lambda$ and $\varepsilon$ is replaced with 
 \begin{equation}\label{w25}
\max_{t\in (0,T)} \int_{G} (w^{(k)})^2\,dx +
\int_{Q} |\nabla w^{(k)}|^2+ \lambda  (w^{(k)})^2\,dx \leq  c_4 k^2 (\mu(L_k))^{1/p'}.
\end{equation} 
     Further, we can repeat the arguments of the previous theorem.  
\end{proof} 

\begin{remark} \label{rem11} The claims of theorems \ref{thr2}, \ref{thr3} are valid also for generalized solutions to the problems \eqref{e1}, \eqref{e2} and   
\eqref{e1}, \eqref{e3}. The arguments are almost the same as those in  \cite[Sect. 1,2 of Ch. 5]{lad}). The only difference is that on the final step we repeat the arguments of the proofs of Theorems \ref{thr2}, \ref{thr3}.  
\end{remark}

\begin{theorem} \label{thr4}
Let  $u\in W_p^{1,2}(Q)$ $(p>(n+2)/2)$ be a solution to the problem  {\rm \eqref{e1}, \eqref{e2}} or the problem {\rm \eqref{e1}, \eqref{e3}}. Assume that the conditions {\rm \eqref{e6}-\eqref{e8}}  in the case of the problem   {\rm \eqref{e1}, \eqref{e2}} or, respectively, the  conditions  
 {\rm \eqref{e6}-\eqref{e8}},  {\rm \eqref{e201}}  in the case of the problem   {\rm \eqref{e1}, \eqref{e3}} are fulfilled. Moreover, there exists 
 $\varepsilon_0>0$ such that $u_0\in C^{2\varepsilon_0}(\overline{G})$, $g\in C^{\varepsilon_0,2\varepsilon_0}(\overline{S})$ in the case of the problem   {\rm \eqref{e1}, \eqref{e2}} and $u_0\in C^{2\varepsilon_0}(\overline{G})$ in the case of the problem   {\rm \eqref{e1}, \eqref{e3}}.
    Then there exists constants  $M_1,\alpha_1>0$, depending on   $\beta_1$, 
$\|g_3\|_{L_{p}(Q)}$, $\|u\|_{L_\infty(Q)}= M$, $\|u_0\|_{C^{2\varepsilon_0}(\overline{G})}$,  and  $\|g\|_{C^{\varepsilon_0,2\varepsilon_0}(\overline{S})}$ in the case of the problem 
  {\rm \eqref{e1}, \eqref{e2}} and  on   $\beta_1$, 
$\|g_3\|_{L_{p}(Q)}$, $\|u\|_{L_\infty(Q)}= M$, $\|u_0\|_{C^{2\varepsilon_0}(\overline{G})}$,   $\beta_3$,  $\|g_4\|_{W_p^{s_0,2s_0}(S)}$ in the case of the problem  {\rm \eqref{e1}, \eqref{e3}}
such that  $\|u\|_{C^{\alpha_1,2\alpha_1}(\overline{Q})}\leq M_1$. 
    \end{theorem}
\begin{proof}
The claim in the case of the problem  {\rm \eqref{e1}, \eqref{e2}} results from Theorem 1.1 of Ch. 5 in \cite{lad}. 
Unfortunately, similar result in the case of the problem 
 {\rm \eqref{e1}, \eqref{e3}} is proven in  \cite{lad} under more stringent conditions on the data  (see Theorem   7.1 of Ch. 5 \cite{lad}).
So we outline the proof which is in line with that in \cite{lad}. Let  $B_r(x_0)$ be a ball of radious  $r$ centered at  $x_0$
and $Q(\rho,\tau)=\{(t,x): x\in B_{\rho}(x_0), \ t_0<t<t_0+\tau\}$.  Fix  $(t_0,x_0)\in \overline{Q}$ and multiply the equation 
 $\eqref{e1}$ by $u^{(k)}\xi^2(t,x)$,  $u^{(k)}=\max(u-k,0)$, 
$k\geq \max_{x\in Q(\rho,\tau)\cap G} u_0(x)$,  
where  $\xi$ is an arbitrary smooth function with values between zero and 1 and vanishing for   $x\in {\mathbb R}^n \setminus  B_\rho(x_0)$ for every  $t$.
Integrating the equality obtained over   $B_{\rho}(x_0)=B_\rho$ and by parts, we infer  
 \begin{multline}\label{e23}
\frac{1}{2}\frac{\partial}{\partial t} \int_{B_\rho\cap G} \xi^2 (u^{(k)})^2\,dx +\int_{B_\rho\cap G} 
\sum_{i,j=1}^n a_{ij}u^{(k)}_{x_j}u^{(k)}_{x_i}\xi^2  + 2\sum_{i,j=1}^n a_{ij}u_{x_j}u^{(k)}\xi \xi_{x_i}+ \\
b(t,x,u,\nabla u) u^{(k)}\xi^2\,dx=-\int_{\Gamma\cap B_\rho}\xi^2 \psi(t,x,u) u^{(k)}\,d\Gamma +\int_{B_\rho} \xi\xi_t(u^{(k)})^2\,dx,  
\end{multline}
Let   $A_k(t)=\{x\in B_\rho\cap G: u(t,x)>k\}$, $L_k=\{(t,x)\in Q(\rho,\tau)\cap Q: u(t,x)>k\}$.   
The conditions of the theorem yield 
\begin{multline*}
\frac{1}{2}\frac{\partial}{\partial t} \int_{B_\rho\cap G}\xi^2 (u^{(k)})^2\,dx +
\int_{B_\rho\cap G} \delta_0\xi^2 |\nabla u^{(k)}|^2 \leq   
\int_{B_\rho\cap G} c_1|\nabla u^{(k)}| |\nabla\xi| \xi u^{(k)}\\ +\beta_1|\nabla u^{(k)}|^2 u^{(k)}\xi^2+g_3\xi^2  u^{(k)} \,dx
+ \int_{\Gamma\cap B_\rho}\xi^2 (g_4+\beta_3|u|)|u^{(k)}|\,d\Gamma + \int_{B_\rho\cap G} |\xi\xi_t|(u^{(k)})^2\,dx.
\end{multline*}
Using the inequality  $|\nabla u^{(k)}|\nabla\xi| \xi u^{(k)}\leq \varepsilon_1 |\nabla u^{(k)}|^2\xi^2+ c(\varepsilon_1)|\nabla \xi|^2( u^{(k)})^2$
for sufficiently small  $\varepsilon_1$, we validate the inequality 
\begin{multline*}
\frac{1}{2}\frac{\partial}{\partial t} \int_{B_\rho\cap G}\xi^2 (u^{(k)})^2\,dx +
\frac{2}{3} \int_{B_\rho\cap G} \delta_0\xi^2 |\nabla u^{(k)}|^2\,dx \leq 
\int_{B_\rho\cap G} \beta_1|\nabla u^{(k)}|^2 u^{(k)}\xi^2+g_3\xi^2  u^{(k)} \,dx \\
\int_{\Gamma\cap B_\rho}\xi^2 (g_4+\beta_3|u|)|u^{(k)}|\,d\Gamma   
 + \int_{B_\rho\cap G} c_1(u^{(k)})^2(|\nabla \xi|^2+\xi|\xi_t|) \,dx.
\end{multline*}
Integrating in $t$ from $t_0$ to $t\leq t_0+\tau$ we obtain the inequality 
\begin{multline}\label{s1}
\max_{t\in t_0,t_0+\tau} \int_{B_\rho\cap G}\xi^2 (u^{(k)}(t,x))^2\,dx +
\frac{2}{3} \int_{L_k} \delta_0\xi^2 |\nabla u^{(k)}|^2\,dxdt \leq \\
\int_{L_k} 2\beta_1|\nabla u^{(k)}|^2 u^{(k)}\xi^2+2g_3\xi^2  |u^{(k)}| \,dxdt +
2\int_{S\cap Q(\rho,\tau)}\xi^2 (g_4+\beta_3|u|)|u^{(k)}|\,dS \\
 + 2\int_{L_k} c_1(u^{(k)})^2(|\nabla \xi|^2+\xi|\xi_t|) \,dxdt +\int_{B_\rho\cap G}\xi^2 (u^{(k)}(t_0,x))^2\,dx.
\end{multline}
We have that  $(g_4+\beta_3|u|)\leq (g_4+\beta_3M)=\tilde{g}_4$.
Arguing as before (see the construction of the function $\rho_1$), 
we can  construct a function $\rho_2$ such that $\rho_{2t}-\Delta \rho_2=0,\  \frac{\partial \rho_2}{\partial \nu}|_S=\tilde{g}_4 $. In this case, we derive the equality 
$$
\int_{Q(\rho,\tau)} \Delta\rho_2 |u_k|\xi^2\,dxdt+ \int_{Q(\rho,\tau)} \nabla \rho_2\nabla   |u_k|\xi^2\,dxdt=\int_{S\cap Q(\rho,\tau)} \tilde{g}_4\xi^2 u^{(k)}\,dS
$$
which ensures  the inequality 
\begin{multline*}
\int_{S\cap Q(\rho,\tau)}\xi^2 (g_4+\beta_3|u|)|u^{(k)}\,dS\leq \int_{L_k}|\Delta \rho_2|\xi^2 |u^{(k)}|+ 2|\nabla\rho_2||\nabla \xi| \xi |u^{(k)}| +
\\ |\nabla\rho_2|\xi^2 |\nabla u^{(k)}|\, dxdt   \leq  \varepsilon \int_{L_k} \xi^2|\nabla u^{(k)}|^2  + 
c(\varepsilon) \int_{L_k}  |u^{(k)}|^2 |\nabla\xi|^2 + (|\Delta \rho_2|+|\nabla \rho_2|^2)\xi^2\,dxdt. 
\end{multline*}
The H\"{o}lder inequality yields 
$$
c(\varepsilon) \int_{L_k} (|\Delta \rho_2|+|\nabla \rho_2|^2)\xi^2\,dxdt\leq c(\int_{L_k}\xi\,dxdt )^{1/p'}. 
$$
So we can write out the estimate
\begin{multline}\label{s2}
\int_{S\cap Q(\rho,\tau)}\xi^2 (g_4+\beta_3|u|)|u^{(k)}|\,dS\leq  \varepsilon \int_{L_k} \xi^2|\nabla u^{(k)}|^2 + \\
c(\varepsilon) \int_{L_k}  |u^{(k)}|^2 |\nabla\xi|^2\,dxdt + c_1(\varepsilon)(\int_{L_k}\xi\,dxdt )^{1/p'}.
\end{multline}
Similarly, we infer 
\begin{equation}\label{s3}
\int_{L_k} 2g_3\xi^2  u^{(k)} \,dxdt\leq c_2(M)  (\int_{L_k}\xi\,dxdt )^{1/p'}
\end{equation} 
Using  \eqref{s2},   \eqref{s3}, and choosing small  $\varepsilon$, we can rewrite   \eqref{s1} in the form 
\begin{multline}\label{s4}
\max_{t\in t_0,t_0+\tau} \int_{B_\rho\cap G}\xi^2 (u^{(k)})^2\,dx +
\frac{1}{3} \int_{L_k} \delta_0\xi^2 |\nabla u^{(k)}|^2\,dxdt \leq 
\int_{L_k} 2\beta_1|\nabla u^{(k)}|^2 u^{(k)}\xi^2\,dxdt \\
+ c_3\int_{L_k} c_1(u^{(k)})^2(|\nabla \xi|^2+\xi|\xi_t|) \,dxdt + c_4(\int_{L_k}\xi\,dxdt )^{1/p'}+\int_{B_\rho\cap G}\xi^2 (u^{(k)}(t_0,x))^2\,dx.
\end{multline}
In what follows,   we assume that   $k\geq \max_{x\in Q(\rho,\tau)\cap G}u(t,x)-\delta$, where $\delta=\delta_0/(12\beta_1)$. This assumption proves that
\begin{multline*}
\max_{t\in t_0,t_0+\tau} \int_{B_\rho\cap G}\xi^2 (u^{(k)})^2\,dx +
\frac{1}{6} \int_{L_k} \delta_0\xi^2 |\nabla u^{(k)}|^2\,dxdt \leq \\
+ c_3\int_{L_k} c_1(u^{(k)})^2(|\nabla \xi|^2+\xi|\xi_t|) \,dxdt + c_4(\int_{L_k}\xi\,dxdt )^{1/p'}+\int_{B_\rho\cap G}\xi^2 (u^{(k)}(t_0,x))^2\,dx.
\end{multline*}
where $k\geq \max_{x\in Q(\rho,\tau)\cap G}u(t,x)-\delta$,  $k\geq \max_{x\in Q(\rho,\tau)\cap G} u_0(x)$. Similar inequality is derived for 
the function  $(-u)^{(k)}$.
The conditions of Theorem 8.2 of Ch. 2 (see Remark  8.1 after this theorem)  in \cite{lad} are fulfilled and, thus, the theorem is proven. 
\end{proof}

\begin{lemma} \label{lem3} Let the condition  {\rm \eqref{e2051}-\eqref{e2053}} be fulfilled.
 If   $u\in W_p^{s_1,2s_1}(S)$ $(s_1=1-1/2p)$ and
  $\|u\|_{L_\infty(S)}=M$, then there exists  $\varepsilon>0$ such that 
  \begin{equation}\label{in1}
\|\psi(t,x,u)\|_{{W}^{s_0,2s_0}_p(S)}\leq   c_1(M)+ c_2(M)\|u\|_{{W}^{s_1-\varepsilon/2,2s_1-\varepsilon}_p(S)},
\end{equation}
 If   $u\in W_p^{1,2}(Q)$ and 
  $\|u\|_{L_\infty(S)}=M$, then there exists  $\varepsilon>0$ such that 
\begin{equation}\label{in2}
  \|a_{ij}u_{x_j}\|_{W_p^{s_0,2s_0}(S)}\leq  c_0\|a_{ij}\|_{L_\infty(S)} \|u\|_{W_p^{1, 2}(Q)}+   
 c_4(M)\|u\|_{W_p^{1-\varepsilon/2, 2-\varepsilon}(Q)},\  u\in W_p^{1,2}(Q),
   \end{equation}
   for all  $i,j$.
   \end{lemma}

\begin{proof}  The proof of the estimate   \eqref{in1} is essentially simpler than that of \eqref{in2}. It relies on the equality
 $$
\varphi(t,x,v_1)-\varphi(t,x,v_2) =\int_0^1\varphi_v(t,x,v_2+\xi(v_1-v_2))\,d\xi(v_1-v_2).
$$
Hence, we proceed with the proof of  \eqref{in2}. 
We employ the definition of the norm  \cite[Subsect. 4.4.1]{tri}, \cite{ama}.
The norm in  $W_p^{s_0,2s_0}(S)$ is defined with the use of 
partition of unity and straightening of the boundary  $\Gamma$. 
As a result, we can prove the estimate in the simplest case of 
$S=(0,\infty)\times \Gamma$,  $\Gamma={\mathbb R}^{n-1}$, $Q=(0,\infty)\times {\mathbb R}^n_+$, ${\mathbb R}^n_+=\{x\in {\mathbb R}^n: \  x_n>0\}$. 
The norm in this case is defined by the equality  (see \cite[Sect. 4.4.2]{tri})
\begin{multline}\label{ee}
\|v(t,x)\|_{W_p^{s_0, 2s_0}(S)}^p=\|v\|_{L_p(S)}^p+  \int_{0}^{\delta} \frac{1}{|\tau|^{1+s_0p}} 
 \|\Delta_{\tau,t} v\|_{L_{p}(S)}^p\,d\tau +\\  \int_{|h|\leq \delta} \frac{1}{|h|^{n-1+2s_0p}} 
 \|\Delta_{h,x} v\|_{L_{p}(S)}^p\,dh=\|v\|_{L_p(S)}^p+ J_1^p(\Delta_{\tau,t} v)+ J_2^p(\Delta_{h,x} v)    ,\ \delta>0,
\end{multline}
where $\Delta_{\tau,t} v=v(t+\tau,x)-v(t,x)$, $\Delta_{h,x} v=v(t,x+h)-v(t,x)$.
Consider the second summand with $v=a_{ij}(t,x,u)u_{x_j}$ assuming that this function is compactly supported.  
We have the equality  $\Delta_{\tau,t}v=a_{ij}(t+\tau,x,u(t+\tau,x))\Delta_{\tau,t}u_{x_j}+(a_{ij}(t+\tau,x, u(t+\tau,x))-a_{ij}(t+\tau,x, u(t,x))) u_{x_j}(t,x)+
 (a_{ij}(t+\tau,x,u(t,x))-a_{ij}(t,x,u(t,x))) u_{x_j}(t,x)=I_1+I_2+I_3$.
The triangle inequality implies that 
$$
J_1(\Delta_{\tau,t} a_{ij} u_{x_j})\leq J_1(I_1)+ J_1(I_2)+J_1(I_3).
$$
Estimate every summand. We have  (see Corollary  1.3 in \cite{pya11})
\begin{multline}\label{in3}
J_1(I_1)\leq \|a_{ij}(t+\tau,x,u(t+\tau,x))\|_{L_\infty(S)}\|u_{x_j}\|_{W_p^{s_0}(0,\infty;L_p(\Gamma))}\leq \\
c_0\|a_{ij}(t+\tau,x,u(t+\tau,x))\|_{L_\infty(S)}\|u\|_{W_p^{1,2}(Q)}.
\end{multline}
Next the H\"{o}lder inequality yields 
\begin{multline}\label{in4}
J_1(I_3)\leq J_1((a_{ij}(t+\tau,x,u(t,x))-a_{ij}(t,x,u(t,x))) u_{x_j}(t,x)\|)\leq \\
 \Bigl(\int_0^\delta \frac{1}{|\tau|^{1+s_0p}}\|u_{x_j}\|_{L_{q_5p/(q_5-p)}(S)}^p\|(a_{ij}(t+\tau,x,u)-a_{ij}(t,x,u)\|_{L_{q_5}(S;C([-M,M]))}^p\,d\tau\Bigr)^{1/p}\leq \\
 \|a_{ij}(t,x,u)\|_{B_{q_5 p}^{s_0,2s_0} (S;C([-M,M]))} \|u_{x_j}\|_{L_{q_5p/(q_5-p)}(S)}\leq c_1(M)\|u_{x_j}\|_{L_{q_5p/(q_5-p)}(S)}\leq \\
 c_2(M) \|u\|_{W_{q_5p/(q_5-p)}^{1/2+\varepsilon/2+1/2\tilde{q}, 1+\varepsilon+1/\tilde{q}}(Q)}^p\leq   \|u\|_{W_{p}^{1-\varepsilon_1, 2-2\varepsilon_1}(Q)}^p,\ \tilde{q}=q_5p/(q_5-p),
 \end{multline}
 where $\varepsilon_1>0$,  $2\varepsilon_1\leq 1-1/p -(n+1)/q_5-\varepsilon$, $0<\varepsilon<1-1/p -(n+1)/q_5$ 
 and we employ Corollary 1.3 in \cite{pya11} and the embedding 
 $W_{p}^{1-\varepsilon_1, 2-2\varepsilon_1}(Q)\subset W_{q_5p/(q_5-p)}^{1/2+\varepsilon/2+1/2\tilde{q},1+\varepsilon+1/\tilde{q}}(Q)$ (Corollary  5.6.4 of Ch. 7 in \cite{ama}). The constant $\varepsilon$ can be chosen to be arbitrarily small. 
Estimate the summand 
$J_1(I_2)$. We use here the inequality of the Gagliardo-Nirenberg type  (Theorem  5.7.1 of Ch. 7 in \cite{ama})
\begin{equation}\label{in5}
\|u\|_{B_{q,q}^{\tilde{s}/2,\tilde{s}}(Q)}\leq c\|u\|_{B_{p,p}^{\tilde{s}_1/2,\tilde{s}_1}(Q)}^{\theta} \|u\|_{B_{\infty,\infty}^{\tilde{s}_0/2,\tilde{s}_0}(Q)}^{1-\theta},
\end{equation}
where   $\tilde{s}-(n+2)/q=(1-\theta)\tilde{s}_0+\theta(\tilde{s}_1-(n+2)/p)$, $(\tilde{s}-\tilde{s}_0)/(\tilde{s}_1-\tilde{s}_0)\leq \theta\leq 1$, $-\infty<\tilde{s}_0<\tilde{s}<\tilde{s}_1<\infty$, 
$1\leq q,p\leq \infty$, $\theta>0$. 
 Since  the function  $a_{ij}$ meets  the Lipschitz condition  in the variable $u$, we infer
 \begin{multline}\label{in6}
J_1(I_2)\leq c(M) J_1(|\Delta_{\tau,t}u(t,x)| |u_{x_j}(t,x)|)\leq  c(M)\|u_{x_j}\|_{L_{q}(S)} 
\|u\|_{B_{qp/(q-p),p}^{{s}_0,2{s}_0}(S)}\leq  \\
c_1(M)\|u_{x_j}\|_{W_{q}^{(1+1/q+\varepsilon)/2,1+1/q+\varepsilon}(Q)} 
\|u\|_{W_{qp/(q-p)}^{{s}_0+(q-p)/2qp+\varepsilon/2,2{s}_0+(q-p)/qp+ \varepsilon}(Q)},
\end{multline}
where  $\varepsilon>0$ is a small parameter, we use the H\"{o}lder inequality, Corollary  1.3 in \cite{pya11}, and 
the embedding  $B_{q,q}^{s+\varepsilon}\subset B_{q,p}^s$ valid for every  $\varepsilon>0$ and $p\in [1,\infty]$.
  To estimate the factors in \eqref{in6}, we  use  \eqref{in5}. 
 Take $\tilde{s}=1+1/q+\varepsilon$, $\tilde{s}_0=3\varepsilon\leq \alpha_1$ (see Theorem  \ref{thr4}),
 $\tilde{s}_1=2-\varepsilon$, $\theta=1/2$.  We can conclude that 
  \begin{equation}\label{in7}
 \|u_{x_j}\|_{W_{q}^{1/2+1/2q+\varepsilon/2,1+1/q+\varepsilon}(G)}\leq c \|u\|_{W_p^{1-\varepsilon/2,2-\varepsilon}(Q)}^{1/2}
 \|u\|_{B_{\infty,\infty}^{\tilde{s}_0/2, \tilde{s}_0}(Q)}^{1/2},
 \end{equation}
 where $q=\frac{2p(n+1)}{(n+2)}, $ and, thereby, 
 $
 1+\frac{1}{q}+\varepsilon-\frac{n+2}{q}= \frac{\tilde{s}_0}{2}+\frac{1}{2} (2-\varepsilon-(n+2)/p $.
 Now take  $\tilde{s}=1-1/2q=s_0+(q-p)/pq$, $\tilde{s}_0=2\varepsilon$,
 $\tilde{s}_1=2-\varepsilon$. For this choice of  $q$, it follows from  \eqref{in5} that 
\begin{equation}\label{in8}
 \|u\|_{W_{qp/(q-p)}^{1/2-1/2q+\varepsilon/2,1-1/q+\varepsilon}(Q)}\leq c \|u\|_{W_p^{1-\varepsilon/2,2-\varepsilon}(Q)}^{1/2}
 \|u\|_{B_{\infty,\infty}^{\tilde{s}_0/2, \tilde{s}_0}(Q)}^{1/2},
 \end{equation}
 where  $1-1/q-(n+2)(1/p-1/q)= \tilde{s}_0/2+ (2-\varepsilon-(n+2)/p)/2. $
 In this case the inequality  \eqref{in6} can be written in the form 
 \begin{equation}\label{in10}
J_1(I_2)\leq C(M) \|u\|_{W_p^{1-\varepsilon/2,2-\varepsilon}(Q)}.  
\end{equation}
The estimate of the expression   $J_2^p(\Delta_{h,x} v)$ is proven by analogy.  
\end{proof}

\begin{remark}  if  $p>n+2$ then it is possible to replace the conditions  \eqref{e2051}-\eqref{e2053} with more natural conditions 
\begin{equation}\label{e251}
a_{ijx_i}\in L_{p}(Q;C([-R,R])), \ a_{iju}\in L_\infty(Q;C([-R,R])) \  \forall R>0, \ i,j=1,2,\ldots,n, 
\end{equation}
\begin{equation}\label{e25}
a_{ij},\psi\in W_p^{s_0,2s_0}(S;C([-R,R])),   \ \psi_u\in L_{q_4}(S;C([-R,R])), \  \ q_4> n+1.
\end{equation}
\end{remark}

\begin{theorem} \label{thr5}
 Assume that  $u\in W_p^{1,2}(Q)$ ($p>(n+2)/2$) is a solution  to the problem  {\rm \eqref{e1}, \eqref{e2}} or the problem  {\rm \eqref{e1}, \eqref{e3}}, and the conditions 
   {\rm \eqref{e6}-\eqref{e8}, \eqref{e2051}- \eqref{e2053}} hold.    
Then there exists a constant  $M_2>0$ depending on the constants  $M, M_1$ in Theorems  {\rm   \ref{thr2}-\ref{thr4}} such that 
\begin{equation}\label{e259}
\|u\|_{W_p^{1,2}(Q)} \leq M_2.
\end{equation}
 \end{theorem}
 
 \begin{proof}  We present the proof in the case of the problem  \eqref{e1}, \eqref{e3}. The problem  \eqref{e1}, \eqref{e2} is simpler and the arguments are the same.
   Let  $N\in {\mathbb N}$ and $h=T/N$. Construct a finite covering of $\Gamma$ by balls  $B_h(x_i)$ of radius  $h$, the parameter $h$ is chosen below. 
 Let us build it up to a covering of $G$. 
Next, we can construct the covering of  $Q$ by domains of the form  $Q_{kl}=((k-5/4)h,(k+1/4)h)\times B_h(x_l)$ ($k=1,2,\ldots, N$). 
Construct also the corresponding partition of unity on $Q$:  $\varphi_{kl}\in C_0^{\infty}(Q_{kl})$.
Without loss of generality we assume that 
$\sum_{k,l: x_l\in \Gamma}\varphi_{kl}(t,x)=1$ for all  $x\in \Gamma, t\in [0,T]$. 
Since the function  $u$ satisfies the estimate from Theorem \ref{thr5} and the functions  $a_{ij}$ are continuous in all variables, for every  
  $\varepsilon>0$, there exists 
$h=T/N>0$ such that  $|a_{ij}(t,x,u(t,x))-a_{ij}(kh,x_l,u(kh,x_l))|<\varepsilon$ for all  $(t,x)\in Q_{kl}$ and all $i,j$, where the constant 
 $h$ depends on  $M_1$ but it is independent of the function $u$ itself.  
Consider the family of problems 
   \begin{equation}\label{a26}
	Mv= v_t-\sum_{i,j=1}^n a_{ij}(kh,x_l,u_{kl})v_{x_jx_i}=f,\
\end{equation}
 \begin{equation}\label{a27}
	v|_{t=0}=v_0,\ \ v|_{S}= g(t,x), 
\end{equation}
or
\begin{equation}\label{a28}
	v|_{t=0}=u_0,\ \ \sum_{i,j=1}^n a_{ij}(kh,x_l,u_{kl})v_{x_j}\nu_i|_{S}=g(t,x),
\end{equation}
where $u_{kl}\in [-M,M]$ is a collection of constants and $M$  is the constant defined in Theorems  \ref{thr2}, \ref{thr3}.
We suppose that 
\begin{equation}\label{a29}
v_0\in W_p^{2-2/p}(G),\  g\in W_p^{k_0,2k_0}, \  f\in L_p(Q), \  \Gamma\in C^2,
\end{equation}
where $k_0=s_1$ in the case of the Dirichlet boundary condition and  $k_0=s_0$ otherwise.  
Applying Theorem  \ref{thr0} we can say that there exist unique  solutions to the problems   \eqref{a26}, \eqref{a27} and \eqref{a26}, \eqref{a28}
and the estimates 
\begin{equation}\label{a30}
 \|v\|_{W_{p}^{1,2}(Q)}\leq
C_{0}(\|u_{0}\|_{W_{p}^{2-2/p}(G)}+
\|f\|_{L_{p}(Q)}+\|g\|_{W_{p}^{k_{0},2k_{0}}(S)})
\end{equation}
hold, where, without loss of generality, we can assume that 
  $C_0$ is indepedent of  $k,l$.
Multiply the equation  \eqref{e1}  and the boundary conditions  \eqref{e3} by $\varphi_{kl}$. The function  $w=u\varphi_{kl}$ is a solution to the problem 
   \begin{multline}\label{e26}
	Mw= w_t-\sum_{i,j=1}^n a_{ij}(t,x,u)w_{x_jx_i}+  \tilde{b}_{kl}=0,\  \tilde{b}_{kl}= -\varphi_{kl}(\sum_{i,j=1}^n  (a_{ijx_i}u_{x_j}+a_{iju}u_{x_j}u_{x_i})- \\  b(x,t,u,\nabla u))  + \sum_{i,j=1}^n  a_{ij}u_{x_j}\varphi_{kl x_i}+ \sum_{i,j=1}^n  a_{ij}u_{x_i}\varphi_{klx_j}+\sum_{i,j=1}^n  a_{ij}u\varphi_{klx_ix_j}, 
\end{multline}
 \begin{equation}\label{e27}
	\sum_{i,j=1}^n a_{ij}(t,x,u)w_{x_j}\nu_i- \sum_{i,j=1}^n a_{ij}(t,x,u)u\varphi _{kl x_j}\nu_i+ \varphi_{kl} \psi(t,x,u)=0, \   \  w(0,x)=\varphi_{kl}u_0.     
\end{equation}
 Rewrite this problem in the form 
   \begin{multline}\label{e28}
	Mw= w_t-\sum_{i,j=1}^n a_{ij}(kh,x_l,u(k h,x_l))w_{x_jx_i}=-\sum_{i,j=1}^n (a_{ij}(kh,x_l,u(kh,x_l))- 
\\ a_{ij}(t,x,u(t,x))) w_{x_jx_i})- \tilde{b}_{kl}=\tilde{f}_{kl},\   w(0,x)=\varphi_{kl}u_0.     
\end{multline}
 \begin{multline}\label{e29}
	\sum_{i,j=1}^n a_{ij}(kh,x_l,u(kh,x_l))w_{x_j}\nu_i=\sum_{i,j=1}^n (a_{ij}(kh,x_l,u(kh,x_l))-a_{ij}(t,x,u(t,x)))w_{x_j}\nu_i+\\  \sum_{i,j=1}^n a_{ij}(t,x,u(t,x)) u\varphi _{klx_j}\nu_i- \varphi_{kl} \psi(t,x,u)=\tilde{g}_{kl},\ 
  w(0,x)=\varphi_{kl}u_0.     
\end{multline}
The estimate  \eqref{a30} yields 
\begin{equation}\label{e30}
 \|w\|_{W_{p}^{1,2}(Q)}\leq
C_{0}(\|\varphi_{kl}u_0 \|_{W_{p}^{2-2/p}(G)}+
\|\tilde{f}_{kl}\|_{L_{p}(Q)}+\|\tilde{g}\|_{W_{p}^{s_{0},2s_{0}}(S)}).
\end{equation}
Estimate every summand on the right-hand side.   
We have that 
\begin{equation}\label{e31}
 \|\sum_{i,j=1}^n (a_{ij}(kh,x_l,u(h,x_l))-a_{ij}(t,x,u(t,x)))w_{x_jx_i}\|_{L_p(Q)}\leq c_0(p)\varepsilon \|w\|_{W_p^{1,2}(Q)}. 
\end{equation} 
 Lemma  1.21 in \cite{pya11} and the conditions on the data imply that
\begin{multline}\label{e32}
 \| \varphi_{kl}\sum_{i,j=1}^n  a_{ijx_i}u_{x_j}-\sum_{i,j=1}^n  a_{ij}u_{x_j}\varphi_{klx_i}- \sum_{i,j=1}^n  a_{ij}u_{x_i}\varphi_{klx_j}-\sum_{i,j=1}^n  a_{ij}u
 \varphi_{kl x_ix_j}\|_{L_p(Q)}\leq \\
 c\|u\|_{W_p^{1-\sigma,2-2\sigma}(Q)}, 
\end{multline} 
 where $\sigma\in (0,1)$ is a parameter defined in embedding theorems.  Next, we have
\begin{equation}\label{e33}
\|\varphi_{kl}(\sum_{i,j=1}^n a_{iju}u_{x_j}u_{x_i}+ b(x,t,u,\nabla u))\|_{L_p(G)}^p\leq c \| \nabla u \|_{L_{2p}(G)}^{2}+ \|g_3\|_{L_p(G)}^p. 
\end{equation} 
We can assume that  $\alpha_1\leq 2/(n+2)$, otherwise we assign  $\alpha_1=2/(n+2)$. 
 Next, we use the inequality  (Lemma  1.14 in \cite{pya11} or Corollary  5.7.3 of Ch. 7 in  \cite{ama})
 $$
 \| \nabla u \|_{L_{2p}(G)}\leq c \| \nabla u \|_{C^{2\alpha_1}(\overline{G})}^{1/2} \|u\|_{W_p^{2-2\alpha_1}(G)}^{1/2}\leq c(M_1) \|u\|_{W_p^{2-2\alpha_1}(G)}^{1/2}.
 $$
 Integrating  \eqref{e33} in $t$ and using  \eqref{e32}, we conclude that 
\begin{equation}\label{e34}
\|\varphi_{kl}(\sum_{i,j=1}^n a_{iju}u_{x_j}u_{x_i}+ b(x,t,u,\nabla u))\|_{L_p(Q)}\leq c_1 \|u \|_{L_p(0,T;W_p^{2-2\alpha_1}(Q))}+ c_2. 
\end{equation} 
The above estimates imply that 
\begin{equation}\label{e35}
\|\tilde{f}_{kl}\|_{L_p(Q)}\leq \varepsilon c_0(p) \|w\|_{W_p^{1,2}(Q)} + c(M,M_1) \|u \|_{W_p^{1-\sigma_1,2-2\sigma_1}(Q)},
\end{equation} 
where  $\sigma_1=\min(\sigma, \alpha_1)$. 
Lemma  \ref{lem3} ensures the inequality 
\begin{multline}\label{e36} 
 \|\sum_{i,j=1}^n a_{ij}(t,x,u(t,x)) u\varphi _{ijx_j}\nu_i- \varphi_{ij} \psi(t,x,u)\|_{W_p^{s_0,2s_0}(S)}\leq \\  c_5(M)+ c_6(M)\|u\|_{W_p^{s_1-\varepsilon_0, 2s_1-2\varepsilon_0}(S)}\leq
 c_5(M)+ c_7 \|u\|_{W_p^{1-\varepsilon_0, 2-2\varepsilon_0}(Q)},
\end{multline}  
with $\varepsilon_0>0$ -a constant.
Moreover, Lemma  \ref{lem3} guarantees the estimate
\begin{multline}\label{e37} 
\|\sum_{i,j=1}^n (a_{ij}(kh,x_l,u(kh,x_l))-a_{ij}(t,x,u(t,x)))w_{x_j}\nu_i\|_{W_p^{s_0,2s_0}(S)}\leq \\ c_1\varepsilon \|w\|_{W_p^{1,2}(Q)} +
c_7(M)  \|w\|_{W_p^{1-\sigma_2,2-2\sigma_2}(Q)},
\end{multline}   
 where $\sigma_2>0$ is a constant. 
 Thus, the estimates  \eqref{e36} and \eqref{e37} validate the inequality 
\begin{equation}\label{e38}
   \|\tilde{g}_{kl}\|_{W_p^{s_0,2s_0}(S)}\leq c_0 \varepsilon \|w\|_{W_p^{1,2}(Q)} + c_8(M,M_1)  \|w\|_{W_p^{1-\sigma_3,2-2\sigma_3}(Q)} + c_9(M,M_1),
\end{equation} 
 where $\sigma_3>0$ is a constant. In view of  \eqref{e30}, \eqref{e35}, \eqref{e38}, we derive that
  \begin{equation}\label{e39}
 \|w\|_{W_{p}^{1,2}(Q)}\leq (c_0(p)+c_1)C_0 \varepsilon \|w\|_{W_p^{1,2}(Q)} + c_9(M,M_1)  \|u\|_{W_p^{1-\sigma_4,2-2\sigma_4}(Q)}+ 
C_{10}(M,M_1), 
\end{equation}
where $\sigma_4>0$ is a constant. Choose   $\varepsilon<1/(2C_0(c_0(p)+c_1))$ and find the corresponding parameter  $h>0$.  We obtain that 
\begin{equation}\label{e40}
 \|w\|_{W_{p}^{1,2}(Q)}\leq  2c_9(M,M_1)  \|u\|_{W_p^{1-\sigma_4,2-2\sigma_4}(Q)}+ 
2 C_{10}(M,M_1). 
\end{equation}
Write out the estimate for a function  $u$
\begin{equation}\label{e41}
 \|u\|_{W_{p}^{1,2}(Q)}\leq \sum_{k,l} \|\varphi_{kl}u\|_{W_{p}^{1,2}(Q)} \leq 
 c_{11}  \|u\|_{W_p^{1-\sigma_4,2-2\sigma_4}(Q)}+ 
 c_{12}(M,M_1),\  
\end{equation}
where the constant  $C_{12}$ depends on  $M,M_1$ and the corresponding norms of the data, $\sigma_4\in (0,1)$.
Next, we employ the interpolation inequality  (Theorem  1.21 in \cite{pya11} or \cite{ama}).
$$
\|u\|_{W_p^{1-\sigma_4,2-2\sigma_4}(Q)}\leq c_{13} \|u\|_{W_p^{1,2}(Q)}^{\theta}\|u\|_{L_p(Q)}^{1-\theta}\leq 
\varepsilon \|u\|_{W_p^{1,2}(Q)}+ c_{14}(\varepsilon),\ 
 \theta=1-\sigma_4.
 $$
where $\varepsilon >0$ is an arbitrary constant. Using this inequality in  \eqref{e41} with  $\varepsilon=1/2C_{11} $, we establish the claim. 
 \end{proof}

Describe the conditions ensuring uniqueness of solutions to our problems: if  $p>n+2$ then, for every  $R>0$ there exist 
nonnegative functions  $g_1\in L_{q_6}(Q)$, $g_2\in L_{q_7}(Q)$ ($q_6\geq (n+2)/2$,
$q_7\geq n+2$) such that 
\begin{multline}\label{e42}
 |b(t,x,u_1,\vec{p}_1)-b(t,x,u_2,\vec{p}_2)|\leq (|u_1-u_2|g_1(t,x) + |\vec{p}_1-\vec{p}_2|g_2(t,x), \\  \forall |u_1|+|\vec{p}_1|+|u_2|+|\vec{p}_2|\leq R;  
\end{multline}
if  $p\in ((n+1)/2, n+2]$ then, for every  $R>0$, there are nonnegative functions  $g_1,g_2$ and constants  $\alpha,\beta\in [1,2]$ such that 
\begin{multline}\label{e421}
 |b(t,x,u_1,\vec{p}_1)-b(t,x,u_2,\vec{p}_2)|\leq (|u_1-u_2|g_1(t,x)(1+(|\vec{p}_1| +|\vec{p}_2|)^\alpha)+\\  g_2(t,x) |\vec{p}_1-\vec{p}_2| (1+ |\vec{p}_1| +|\vec{p}_2|)^\beta), \  \forall |u_1|+|u_2|\leq R,  \ \vec{p}_1,\vec{p}_2\in {\mathbb R}^n,
\end{multline}
where $g_1\in L_{q_8}(Q)$, $q_8\geq (n+2)p/((2+\alpha)p-\alpha(n+2))$, $g_2\in L_{q_9}(Q)$,  $q_9\geq p(n+2)/(p(1+\beta)-\beta (n+2))$. 

Expose the consistency  conditions.
In the case of the problem  {\rm \eqref{e1}, \eqref{e2}} we require that 
\begin{equation}\label{e43}
g(0,x)=u_0(x)|_\Gamma
\end{equation}
and in the case of the problem  {\rm \eqref{e1}, \eqref{e3}} that   
\begin{equation}\label{e44}
 \sum_{i,j=1}^na_{ij}(0,x,u_0(x))u_{0x_j}\nu_i+ \psi(0,x,u_0)|_{S}=0,\  \textrm{if}\  p>3.
\end{equation}

\begin{theorem} \label{thr6}
 If the conditions 
   {\rm \eqref{e6}-\eqref{e8}, \eqref{e2051}, \eqref{e43}} hold then there exists a solution 
       $u\in W_p^{1,2}(Q)$ ($p>(n+2)/2$) to the problem  {\rm \eqref{e1}, \eqref{e2}}
       Under the conditions  {\rm \eqref{e6}-\eqref{e8}, \eqref{e2051}-\eqref{e2053}, \eqref{e44}},  $p>(n+2)/2$, $p\neq 3$, there exists a solution $u\in W_p^{1,2}(Q)$ to the problem 
         {\rm \eqref{e1}, \eqref{e3}}.    
If the conditions  \eqref{e42}, \eqref{e421} are satisfied then a solution to these problems is defined  uniquely.   
 \end{theorem}
 
\begin{proof}
To prove existence theorems, we employ the conventional scheme exposed in  Sect.  5,6 of Ch. 5 in  \cite{lad}. 
Consider the collections of problems depending on a parameter  $\gamma\in [0,1]$:
\begin{equation}\label{e45}
	Mu= u_t-\sum_{i,j=1}^n \partial_{x_i}((\gamma a_{ij}(t,x,u)+(1-\gamma)\delta_{ij})u_{x_j})+ \gamma b(x,t,u,\nabla u)=0,
\end{equation}
In the former case the initial-boundary conditions are written in the form 
\begin{equation}\label{e46}
	u|_{t=0}=u_0,\ \ u|_{S}= g(t,x), 
\end{equation}
and in the latter in the form 
\begin{equation}\label{e47}
	u|_{t=0}=u_0,\ \  \sum_{i,j=1}^n (\gamma a_{ij}(t,x,u)+(1-\gamma)\delta_{ij}) u_{x_j}\nu_i+ \gamma \psi(t,x,u)-(1-\gamma)\sum_{j=1}^n u_{0x_j}\nu_j  |_{S}=0.  
\end{equation}
For instance, consider the family   \eqref{e45}, \eqref{e47}.
As is easily seen, the conditions  \eqref{e7}-\eqref{e8} are fulfilled for the functions  $\tilde{a}_{ij}=\gamma a_{ij}(t,x,u)+(1-\gamma)\delta_{ij}$
 with the same constants and functions  $g_i$ and the constant  $\delta_0$ in \eqref{e7} is replaced with  $\min(1,\delta_0)$.
The new functions $\tilde{a}_{ij}$, $\tilde{\psi}=\gamma \psi(t,x,u)-(1-\gamma)\sum_{j=1}^n u_{0x_j}\nu_j$ also satisfy 
 \eqref{e2051}-\eqref{e2053} and without loss of generality we can assume that 
 the norms of all  functions in these conditions are estimated by the constants independent of 
 $\gamma$. 
Thus, we can assume that the estimate
\begin{equation}\label{e48}
\|u\|_{W_p^{1,2}(Q)}\leq C(M,M_1),\  \gamma\in [0,1],
\end{equation}
holds and the constant   $C(M,M_1)$ is independent of  $\gamma\in [0,1]$. 
 Construct a function  $\Psi\in W_p^{1,2}(Q)$ such that  $\Psi|_{t=0}=u_0(x)$ and make the change of variables   $u=v+\Psi$. 
In order to construct this function, we can extend  $u_0$ to the whole  ${\mathbb R}^n$ preserving the class  (see Sect. 4.2.2, 4.2.3 in \cite{tri}) and find a solution  
to the Cauchy problem  $\Psi_t-\Delta \Psi=0$, $\Psi|_{t=0}=u_0$ (see Theorem 5.7 in  \cite{denk}).
We arrive at the problem 
\begin{multline}\label{e49}
	M_0v= v_t-\sum_{i,j=1}^n \partial_{x_i}(\tilde{a}_{ij}(t,x,v+\Psi)v_{x_j})- \\ \sum_{i,j=1}^n \partial_{x_i}(\tilde{a}_{ij}(t,x,v+\Psi)\Psi_{x_j}) + {b}(x,t,v+\Psi,\nabla v+\Psi)+\Psi_t=0,
\end{multline}
\begin{equation}\label{e50}
	v|_{t=0}=0,\ \  \sum_{i,j=1}^n  \tilde{a}_{ij}(t,x,v+\Psi) v_{x_j}+ \sum_{i,j=1}^n  \tilde{a}_{ij}(t,x,v+\Psi) \Psi_{x_j} + {\psi}(t,x,v+\Psi)|_{S}=0.  
\end{equation}
Denote by  $\Phi(\gamma,w)$ a solution to the  problem 
\begin{multline}\label{e51}
	M_0v= v_t-\sum_{i,j=1}^n \partial_{x_i}(\tilde{a}_{ij}(t,x,w+\Psi)v_{x_j})= \sum_{i,j=1}^n \partial_{x_i}( \tilde{a}_{ij}(t,x,w+\Psi)\Psi_{x_j}) \\ - {b}(x,t,w+\Psi,\nabla w+\Psi)-\Psi_t =f(w),\  \ v|_{t=0}=0,\
\end{multline}
\begin{equation}\label{e52}
	  \sum_{i,j=1}^n  \tilde{a}_{ij}(t,x,w+\Psi) v_{x_j}=- \sum_{i,j=1}^n  \tilde{a}_{ij}(t,x,w+\Psi) \Psi_{x_j}-  {\psi}(t,x,w+\Psi)|_{S}=g(w),  
\end{equation}
where  $w\in H_1=\{w\in W_{p}^{1-\varepsilon_0,2-2\varepsilon_0}(Q): \  w(0,x)=0\}$ and  $\varepsilon_0$ is a positive number less than the minimum of the constants 
 $\varepsilon$ in \eqref{in1}, \eqref{in2} and the number $1/2-(n+2)/4p$. Under this condition  $W_{p}^{1-\varepsilon_0,2-2\varepsilon_0}(Q)\subset W_{2p}^{1/2,1}$ 
(see Corollary  5.6.4 of Ch.  7 in \cite{ama}) or Theorem  1.22 in \cite{pya11}). Moreover, we can assume  decreasing $\varepsilon_0 $ if necessary that 
$W_p^{1-\varepsilon_0,2-2\varepsilon_0}(Q)\subset B_{pq_4/(q_4-p) p}^{s_0+1/2p-1/2q+\varepsilon, 2s_0+1/p-1/q+2\varepsilon}(Q) $ (Corollary  5.6.4 of Ch. 7 in \cite{ama})
and thus we have the estimate 
\begin{equation}\label{a52}
\|v\|_{B_{pq_4/(q_4-p) p}^{s_0+\varepsilon, 2s_0+2\varepsilon}(S)}\leq \|v\|_{B_{pq_4/(q_4-p) p}^{s_0+1/2p-1/2q+\varepsilon, 2s_0+1/p-1/q+2\varepsilon}(Q)}\leq 
c_1\|v\|_{W_p^{1-\varepsilon_0,2-2\varepsilon_0}(Q)}\ \forall v\in H_1. 
\end{equation}
 Endow the space  $H_1$ with the norm coinciding with the norm in $W_{p}^{1-\varepsilon_0,2-2\varepsilon_0}(Q)$.  If 
  $w\in H_1$ then the right-hand sides in \eqref{e51},  \eqref{e52} belong to  $L_p(Q)$, $W_p^{s_0,2s_0}(S)$, respectively,  (see Lemma \ref{lem3} and the proof of Theorem 
   \ref{thr4}). 
Note that 
\begin{multline}\label{e53} 
  W_{2p}^{1/2,1}(Q)\subset   C^{1/2-(n+2)/4p, 1-(n+2)/2p}(\overline{Q}),\\ u\in W_{2p}^{1/2,1}(Q) \Rightarrow
u|_{S}\in W_{2p}^{1/2-1/4p,1-1/2p}(S)\subset  B_{q_3p}^{s_0}(0,T; L_{q_3}(\Gamma))\cap L_{q_3}(0,T; B_{q_3p}^{2s_0}(\Gamma)).
  \end{multline}
By Theorem  \ref{thr0}, there exists a unique solution 
 to the problem  \eqref{e51},  \eqref{e52} of the class  $W_p^{1,2}(Q)$. 
 The following inequality is valid: 
 $$
 \tilde{\delta}_0|\xi|^2\leq \sum_{i,j=1}^n \tilde{a}_{ij}\xi_i\xi_j \leq \frac{1}{\tilde{\delta}_0}|\xi|^2,
 $$
 where  $\tilde{\delta}_0=\min(\delta_0,\delta_1)$,  $1/\delta_1=\max(1/\delta_0,\sup_{|w|\leq R_0}\sup_{|\xi|=1}\sum_{i,j=1}^n \tilde{a}_{ij}(t,x,w)\xi_i\xi_j)$, $R_0=\|w\|_{C(\overline{Q})}$.
 In this case, we infer 
 \begin{equation}\label{e54}
 \|v\|_{W_{p}^{1,2}(Q)}\leq
C_{0}(\|u_{0}\|_{W_{p}^{2-2/p}(G)}+
\|f(w)\|_{L_{p}(Q)}+\|g(w)\|_{W_{p}^{s_{0},2s_{0}}(S)}),
\end{equation}
where the constant $C_0$ depends on   $\tilde{\delta}_0$ and the norms of  $\tilde{a}_{ij}$.
Since the embedding  $W_{p}^{1,2}(Q)\subset W_{2p}^{1/2,1}(Q)$  \cite[Theorem 7.5.2, Ch. 7]{ama} is compact,  the estimate \eqref{e54} implies that the mapping 
$w\to \Phi(\gamma,w) $ is compact as well. Demonstrate that it is continuous. Let a sequence  $w_n\in H_1$ converges to $w\in H_1$ in the norm of the space  $H_1$.
Consider the sequence  $b_n=b(t,x,w_n+\Psi,\nabla (w_n+\Psi))$. In view of the condition  \eqref{e8}, we can conclude that 
 \begin{equation}\label{e55}
 \|b_n\|_{L_{p}(Q)}\leq c_1+ c_2(\|w_n\|_{L_{2p}(Q)}^2+ \|\nabla w_n\|_{L_{2p}(Q)}^2)\leq (c_3+c_4R_1),\  R_1=\max_n\|w_n\|_{H_1}, 
\end{equation}
Assign   $R_2=\max_n \|w_n+\Psi\|_{C^{1/2-(n+2)/4p, 1-(n+2)/2p}(\overline{Q})}$.  In view of  \eqref{e53}, $R_2<\infty$. 
By Lemma 3.1 of Ch.  2 in \cite{lad}, there exists a subsequence  $w_{n_k}$ such that  $w_{n_k}\to w$, $\nabla w_{n_k}\to \nabla w$ a.e. in $Q$. 
By Theorem  4.9  in \cite{bre}, choosing one more subsequence if necessary, we can say that 
there exists a function  $\psi\in L_{2p}(Q)$ such that  $|w_{n_k}|+|\nabla w_{n_k}|\leq \psi(t,x)$ a.e.  
In this case  $b_{n_k}-b\to 0$ a.e. and $|b_{n_k}-b|^p\leq 2g_3^p+C_2 |\psi|^{2p} $ a.~e.  
The Lebesgue dominated convergence theorem implies that  $\|b_{n_k}-b\|_{L_{p}(Q)}\to 0$ as  $k\to \infty$. 
Consider the second function  $l(w_n)= \sum_{i,j=1}^n \partial_{x_i}( \tilde{a}_{ij}(t,x,w_n+\Psi)\Psi_{x_j})$ occurring into  $f(w_n)$.
We have that 
 \begin{multline*}
 l(w_n)= \sum_{i,j=1}^n  (\tilde{a}_{ijx_i}(t,x,w_n+\Psi)\Psi_{x_j}+\tilde{a}_{iju}(t,x,w_n+\Psi)(w_n+\Psi)_{x_i}\Psi_{x_j}\\
 +\tilde{a}_{ij}(t,x,w_n+\Psi)\Psi_{x_ix_j})=I_1(w_n)+I_2(w_n)+I_3(w_n).
 \end{multline*}
 Since the functions  $\tilde{a}_{ij}$ are continuous and   $\|w_n-w\|_{C^{1/2-(n+2)/4p, 1-(n+2)/2p}(\overline{Q})}\to 0$ (see \eqref{e53}),  
 $\|I_3(w_n)-I_3(w)\|_{L_p(Q)}\to 0 $ as  $n\to \infty$. 
 Moreover,  we have that 
 \begin{multline*}
 I_2(w_{n})-I_2(w)=\sum_{i,j=1}^n (
 \frac{1}{2}(\tilde{a}_{iju}(t,x,w_n+\Psi)+\tilde{a}_{iju}(t,x,w+\Psi))(w_n-w)_{x_i}\Psi_{x_j}+\\
 \frac{1}{2}(\tilde{a}_{iju}(t,x,w_n+\Psi)-\tilde{a}_{iju}(t,x,w+\Psi))(w_n+w)_{x_i}\Psi_{x_j})
 \end{multline*}
 This representation ensures the estimate 
 \begin{multline*}
 \|I_2(w_{n})-I_2(w)\|_{L_p(Q)}\leq c \|\nabla (w_n-w)\|_{L_{2p}(Q)}\|\nabla \Psi\|_{L_{2p}(Q)}+ \\
 \sum_{i,j=1}^n \|(\tilde{a}_{iju}(t,x,w_n+\Psi)-\tilde{a}_{iju}(t,x,w+\Psi))\Psi_{x_j}\|_{L_{2p}(Q)}\|\nabla (w_n+w)\|_{L_{2p}(Q)}\to 0 \  \textrm{as}\ n\to \infty,
 \end{multline*}
 where the last summand tends to 0 by the Lebesgue dominated convergence theorem. 
 At last, we derive that  
 \begin{multline*}
 \|I_1(w_{n_k})-I_1(w)\|_{L_p(Q)}\leq c\sum_{i,j=1}^n \|(\tilde{a}_{ijx_i}(t,x,w_{n_k}+\Psi)-\tilde{a}_{ijx_i}(t,x,w+\Psi))\Psi_{x_j}\|_{L_{p}(Q)}\leq \\
 C\sum_{i,j=1}^n \|\tilde{a}_{ijx_i}(t,x,w_{n_k}+\Psi)-\tilde{a}_{ijx_i}(t,x,w+\Psi)\|_{L_{q_3}(Q)}\to 0 \  \textrm{as}\ k\to \infty
 \end{multline*}
  also in view of the Lebesgue dominated convergence theorem. Indeed,   $\tilde{a}_{ijx_i}(t,x,w_{n_k}+\Psi)-\tilde{a}_{ijx_i}(t,x,w+\Psi))\to 0$ a.e. in $Q$.
 On the other hand,   $|\tilde{a}_{ijx_i}(t,x,w_{n_k}+\Psi)-\tilde{a}_{ijx_i}(t,x,w+\Psi)|\leq 2\|\tilde{a}_{ijx_i}(t,x,w)\|_{C([-R_2,R_2])}$
 and the last function is integrable in view of \eqref{e251}.
 Finally we can say that  $\|f(w_{n_k})-f(w)\|_{L_p(Q)}\to 0$ as $k\to \infty$. 
 Now, we study the question of  convergence of the function 
 $ 
g(w_n) = - \sum_{i,j=1}^n  \tilde{a}_{ij}(t,x,w_n+\Psi) \Psi_{x_j}-  {\psi}(t,x,w_n+\Psi)|_{S}=J_1(w_n)+J_2(w_n).
$
Estimate  $
\|J_2(w_{n_k})-J_2(w)\|_{W_p^{s_0,2s_0}(S)}. 
$
It suffices to establish the necessary estimate in the simplest case 
$S=(0,\infty)\times \Gamma$,  $\Gamma={\mathbb R}^{n-1}$, $Q=(0,\infty)\times {\mathbb R}^n_+$, ${\mathbb R}^n_+=\{x\in {\mathbb R}^n: \  x_n>0\}$ assuming that 
all functions are compactly supported. The norm in this case is defined by the equality  \eqref{ee}.
We have that  $\Delta_{\tau,t}(J_2(w_{n_k})-J_2(w))=[\psi(t+\tau,x,w_{n_k}(t+\tau,x)+\Psi(t+\tau,x))- \psi(t,x,w_{n_k}(t+\tau,x)+\Psi(t+\tau,x))-\psi(t+\tau,x,w(t+\tau,x)+\Psi(t+\tau,x))+ 
 \psi(t,x,w(t+\tau,x)+\Psi(t+\tau,x))]+[\psi (t,x, w_{n_k}(t+\tau,x)+\Psi(t+\tau,x))-\psi(t,x, w_{n_k}(t,x)+\Psi(t,x))-\psi (t,x, w(t+\tau,x)+\Psi(t+\tau,x))+\psi(t,x,w(t,x)+\Psi(t,x))]
 =I_{1k}+I_{2k}$. The expression  $\Delta_{h,x}(J_2(w_{n_k})-J_2(w))$ can be represented similarly.
The triangle inequality yields 
\begin{multline*}
 \int_{0}^{\delta} \frac{1}{|\tau|^{1+s_0p}} 
 \|\Delta_{\tau,t} (J_2(w_{n_k})-J_2(w))\|_{L_{p}(S)}^p\,d\tau\leq c(p)\Bigl( \int_{0}^{\delta} \frac{1}{|\tau|^{1+s_0p}} 
 \|I_{1k}(\tau,t,x)\|_{L_{p}(S)}^p\,d\tau+\\  \int_{0}^{\delta} \frac{1}{|\tau|^{1+s_0p}} 
 \|I_{2k}(\tau,t,x)\|_{L_{p}(S)}^p\,d\tau\Bigr)=c(p)(J_{1k}+J_{2k}). 
\end{multline*}
We have that $I_{1k}\to 0$ as $k\to\infty$ a.e. on $S\times (0,\delta)$ and $|I_1|\leq c\|\Delta_{\tau,t}\psi(t,x,w)\|_{C([-R_2,R_2])}$ and the right-hand side here to the power $p$ with the weight
$1/|\tau|^{1+s_0p}$  is integrable. 
The Lebesgue dominated converge theorem ensures that  $J_{1k}\to 0 $ as $k\to \infty$. 
Next, we derive that 
\begin{multline*}
I_{2k}(\tau,t,x)=\int_0^1  \psi_u(t,x, w_{n_k}(t,x)+\Psi(t,x)+r(w_{n_k}(t+\tau,x)-w_{n_k}(t,x)+\Psi(t+\tau,x)- \\ \Psi(t,x)))dr (w_{n_k}(t+\tau,x)-w_{n_k}(t,x)+\Psi(t+\tau,x)-\Psi(t,x))-
\int_0^1  \psi_u(t,x, w(t,x)+\Psi(t,x)+\\ r(w(t+\tau,x)-w(t,x)+  \Psi(t+\tau,x)-\Psi(t,x)))dr (w(t+\tau,x)-w(t,x)+\Psi(t+\tau,x)-\Psi(t,x)).
\end{multline*}
This equality can be rewritten in the form 
\begin{multline*}
I_{2k}(\tau,t,x)=\frac{1}{2}\int_0^1  \psi_u(t,x, w_{n_k}(t,x)+\Psi(t,x)+r(\Delta_{\tau,t}(w_{n_k}(t,x)+\Psi(t,x)))) + \\ 
\psi_u(t,x, w(t,x)+  \Psi(t,x)+ r(\Delta_{\tau,t}(w(t,x)+  \Psi(t,x)))) dr   \Delta_{\tau,t}(w_{n_k}(t,x)-w(t,x))+   
  \\ \frac{1}{2}\int\limits_0^1  \psi_u(t,x, w_{n_k}(t,x)+ \Psi(t,x)+r(\Delta_{\tau,t}(w_{n_k}(t,x)+\Psi(t,x)))) -   
\psi_u(t,x, w(t,x)+\Psi(t,x) \\ + r(\Delta_{\tau,t}(w(t,x)+  \Psi(t,x)))) dr   \Delta_{\tau,t}(w_{n_k}(t,x)+w(t,x)+ 2\Psi(t,x))=\\
\frac{1}{2}\int\limits_0^1  \tilde{I}_{1k}(\tau,t,x,r)\,dr     \Delta_{\tau,t}(w_{n_k}(t,x)-w(t,x)) +  
 \frac{1}{2}\int\limits_0^1  \tilde{I}_{2k}(\tau,t,x,r)\,dr      \Delta_{\tau,t}(w_{n_k}(t,x)+\\ w(t,x)+ 2\Psi(t,x))  =
 \frac{1}{2}\int\limits_0^1  (\tilde{I}_{1k}(\tau,t,x,r)+\tilde{I}_{2k}(\tau,t,x,r))\,dr     \Delta_{\tau,t}(w_{n_k}(t,x)-w(t,x)) +  \\
 \frac{1}{2}\int\limits_0^1  \tilde{I}_{2k}(\tau,t,x,r)\,dr     \Delta_{\tau,t}(2w(t,x)+ 2\Psi(t,x)).
 \end{multline*}
In view of  \eqref{a52}, 
\begin{equation}\label{a53}
\|w_{n_k}-w\|_{B_{pq_4/(q_4-p) p}^{s_0}(S)}\leq c \|w_{n_k}-w\|_{H_1} \to 0\ \textrm{as}\  k\to \infty. 
\end{equation}
The H\"{o}lder inequality  implies that 
\begin{multline*}
J_{2k}\leq c \|\psi_u (t,x,w)\|_{L_{q_4}(S;C([-R_2,R_2]))}\|w_{n_k}-w\|_{B_{pq_4/(q_4-p) p}^{s_0}(S)}+  \\
c_1\int_0^\delta \frac{1}{\tau^{1+s_0p}}\Bigl( \int_0^1\int_S |\tilde{I}_{2k}|^{q_4}\,dSdr\Bigr)^{p/q_4}  \|\Delta_{\tau,t}(w+\Psi)\|_{L_{pq_4/(q_4-p) p}(S)}^p\,d\tau.
\end{multline*} 
The former summand on the right-hand side tend to zero as $k\to \infty$  in view of \eqref{a53}. 
 The Lebesgue dominated converge theorem ensures that  $ \int_0^1\int_S |\tilde{I}_{2k}|^{q_4}\,dSdr\to 0$ as $k\to \infty$. Indeed,  $|\tilde{I}_{2k}|^{q_4}\leq 
\|\psi_u (t,x,w)\|^{q_4}_{C([-R_2,R_2])}$ and the last function is integrable. Moreover, for all $r\in (0,1), \tau>0$ and almost all  $(t,x)$ $I_{2k}\to 0 $ as $k\to \infty$. 
Again, the Lebesgue dominated convergence theorem and the estimate 
\begin{multline*}
\Bigl( \int_0^1\int_S |\tilde{I}_{2k}|^{q_4}\,dSdr\Bigr)^{p/q_4}  \|\Delta_{\tau,t}(w+\Psi)\|_{L_{pq_4/(q_4-p) p}(S)}^p\leq \\ c_2 \|\psi_u (t,x,w)\|_{L_{q_4}(Q;C^([-R_2,R_2]))}^{p}
 \|\Delta_{\tau,t}(w+\Psi)\|_{L_{pq_4/(q_4-p) p}(S)}^p
\end{multline*} 
ensures that the second integral tend to 0 as  $k\to \infty$. Thus,  $J_{2k}\to 0$ as $k\to \infty$. 
We have proven that  $J_2(w_{n_k})-J_2(w)\to 0$ as $k\to \infty$ in $W_p^{s_0}(0,T;L_p(\Gamma))$.
Similar arguments prove that  $J_2(w_{n_k})-J_2(w)\to 0$ as $k\to \infty$ in $L_p(0,T;W_p^{2s_0}(\Gamma))$.
Finally,  $\|\psi(t,x,w_{n_k}+\Psi)  - \psi(t,x,w+\Psi)\|_{W_p^{s_0,2s_0}(S)}\to 0$ as $k\to \infty$. 
Similar arguments are employed in the proof of the convergence 
$\sum_{i,j=1}^n  \tilde{a}_{ij}(t,x,w_n+\Psi) \Psi_{x_j}\to \sum_{i,j=1}^n  \tilde{a}_{ij}(t,x,w+\Psi) \Psi_{x_j}$ in the space $W_p^{s_0,2s_0}(S)$.

Denote by  $v_n,v$ the solutions to the problem \eqref{e51}, \eqref{e52} relating to the  $w_n,w$. 
 In this case the difference  $\varphi_{k}=v_{n_k}-v$ is a solution to the problem 
 \begin{multline}\label{e56}
	M_0\varphi_n= \varphi_{kt}-\frac{1}{2}\sum_{i,j=1}^n \partial_{x_i}((\tilde{a}_{ij}(t,x,w_{n_k}+\Psi)+ \tilde{a}_{ij}(t,x,w+\Psi))\varphi_{kx_j})= \\
\frac{1}{2}\sum_{i,j=1}^n \partial_{x_i}((\tilde{a}_{ij}(t,x,w_{n_k}+\Psi) - \tilde{a}_{ij}(t,x,w+\Psi))(v_{nx_j}+v_{x_j})) + f(w_n)-f(w),
\end{multline}
The right-hand side can be written as  $\tilde{f}(w_{n_k})-\tilde{f}(w)$.
\begin{multline}\label{e57}
	\varphi_k|_{t=0}=0,\ \  \sum_{i,j=1}^n  \frac{1}{2} (\tilde{a}_{ij}(t,x,w_{n_k}+\Psi)+ \tilde{a}_{ij}(t,x,w+\Psi)) \varphi_{kx_j}= -  
\sum_{i,j=1}^n  \frac{1}{2} (\tilde{a}_{ij}(t,x,w_{n_k}+\Psi)\\ - \tilde{a}_{ij}(t,x,w+\Psi)) (v_{n_kx_j}+v_{x_j})+ g(w_n)-g(w)=\tilde{g}(w_n)-\tilde{g}(w) .  
\end{multline}
The functions  $v_n$ meet the estimate  \eqref{e54} and, hence, the norms  $\|v_n\|_{W_p^{1,2}(Q)}$ are bounded uniformly in  $n$. 
In this case the difference  $\varphi_k$ also satisfies the estimate  $\eqref{e54}$ with the right-hand side  $c(\|\tilde{f}(w_{n_k})-\tilde{f}(w)\|_{L_p(Q)}+
\|\tilde{g}(w_{n_k})-\tilde{g}(w)\|_{W_p^{s_0,2s_0}(S)})$. As in the proofs of the convergence   $f(w_{n_k})\to f(w)$, $g(w_{n_k})\to g(w)$, we can demonstrate that 
 $\tilde{f}(w_{n_k})\to \tilde{f}(w)$, $\tilde{g}(w_{n_k})\to \tilde{g}(w)$.
Thus, the convergence  $w_n\to w$ in $H_1$ implies that there exists a subsequence  $w_{n_k}$ such that 
$\Phi(\gamma, w_{n_k})\to \Phi(\gamma, w)$. In this case   $\Phi(\gamma, w_{n})\to \Phi(\gamma, w)$ as $n\to \infty.$ 
We have proven that the mapping  $w\to \Phi(\gamma, w)$ is continuous and compact, its continuity in the parameter 
 $\gamma\in [0,1]$ is obvious.
 
As is easily seen,  $v$ is a solution to the problem  \eqref{e51}, \eqref{e52} if and only if  $v$ is a fixed point of the mapping 
$\Phi(\gamma,w)$, i.~e. 
\begin{equation}\label{e58}
	v=\Phi(\gamma,v).  
\end{equation}
For  $\gamma=0$, a solution to the problem  \eqref{e47}, \eqref{e48} exists and  $v+\Psi=u$ is a solution to the problem   $u_t-\Delta u=0$, $u|_{t=0}=u_0$,
$\frac{\partial u}{\partial \nu}=\frac{\partial u_0}{\partial \nu}$. 
Moreover, we have estimates uniform on the parameter  $\gamma$. By Theorem  37.6 in \cite{kra}
the equation  \eqref{e58} has a solution for all $\gamma\in [0,1]$. 
Now we prove uniqueness of solutions to our problem 
{\rm \eqref{e1}, \eqref{e3}}. Let  $u_1,u_2$ be two solutions of the problem. Subtracting equations  \eqref{e1} for $u_1$ and $u_2$, multiplying 
the inequality obtained  by 
$v=u_1-u_2$, and integrating over  $G$, we derive that 
\begin{multline}\label{e59}
\partial_t \int_G \frac{v^2}{2} +
\frac{1}{2}\sum_{i,j=1}^n (a_{ij}(t,x,u_1)+ a_{ij}(t,x,u_2)) v_{x_j}v_{x_i}\,dx  =-\int_{\Gamma}(\psi(t,x,u_1)-\psi(t,x,u_2))v\,d\Gamma
\\ + \int_G\frac{1}{2}\sum_{i,j=1}^n ({a}_{ij}(t,x,u_1)- \tilde{a}_{ij}(t,x,u_2))(u_{1x_j}+u_{2x_j})v_{x_i}  -(b(t,x,u_1)-b(t,x,u_2))v\,dx
\end{multline}
Fix  $t\in (0,T)$.
Integrating  from 0 to  $\tau\leq t$ and using the conditions on the coefficients, we arrive at the inequality
\begin{multline}\label{e60}
\int_G \frac{v^2}{2}(\tau,x) \,dx+ \int_0^\tau\int_G\delta_0|\nabla v(\xi,x)|^2\,dxd\xi \leq
\int_0^\tau \int_{\Gamma}|(\psi(\xi,x,u_1)-\psi(\xi,x,u_2))v|\,d\Gamma d\xi
\\ + c_1\int_0^\tau\int_G|\nabla (u_{1}+u_{2})| |v||\nabla v| +|(b(\xi,x,u_1)-b(\xi,x,u_2))v|\,dGd\xi\leq \\
 \int_0^t \int_{\Gamma}|(\psi(\xi,x,u_1)-\psi(\xi,x,u_2))v|\,d\Gamma d\xi \\
+ c_1\int_0^t\int_G|\nabla (u_{1}+u_{2})| |v||\nabla v| +|(b(\xi,x,u_1,\nabla u_1)-b(\xi,x,u_2,\nabla u_2))v|\,dGd\xi.
\end{multline}
Taking the maximum in  $\tau$ on the left-hand side, we conclude that  
\begin{multline}\label{e61}
 \max_{\tau\in [0,t]}\int_G \frac{v^2(\tau,x)}{2} \,dx+ \int_0^t\int_G\delta_0|\nabla v|^2\,dxd\xi\leq
\int_0^t \int_{\Gamma}|(\psi(\xi,x,u_1)-\psi(\xi,x,u_2))v|\,d\Gamma d\xi
\\ + c_1\int_0^t\int_G|\nabla (u_{1}+u_{2})| |v||\nabla v| +|(b(\xi,x,u_1,\nabla u_1)-b(\xi,x,u_2,\nabla u_2))v|\,dGd\xi. 
\end{multline}
Estimate every summand on the right-hand side. Let $R=\|u_1\|_{C(S)}+\|u_2\|_{C(S)}$. 
We have
\begin{multline}\label{e62}
\int\limits_0^t \int\limits_{\Gamma}|(\psi(\xi,x,u_1)-\psi(\xi,x,u_2))v|\,d\Gamma d\xi=
\int\limits_0^t \int\limits_{\Gamma}\Bigl|\int_0^1\psi_u(\xi,x,u_1+s(u_2-u_1))\,ds\Bigr| |v|^2\,d\Gamma d\xi \\ \leq  \int\limits_0^t \|\psi(\xi,x,u)\|_{L_{q_4}(\Gamma;C^([-R,R]))}\|v\|_{L_{2 q_4'}(\Gamma)}^2\, d\xi.
\end{multline}
The embedding theorems  \cite[Lemma 1.14]{pya11} and the trace theorems  \cite[Lemma 1.9]{pya11}  validate the estimate 
\begin{equation}\label{e63}
\|v\|_{L_{2 q_4'}(\Gamma)}^2\leq c \|v\|_{W_{2}^s(\Gamma)}^2\leq \|v\|_{W_{2}^{s+1/2}(G)}^2\leq c_2 \|v\|_{W_{2}^{1}(G)}^{2(s+1/2)}\|v\|_{L_{2}(G)}^{2(s-1/2)},
\end{equation}
where $s=(n-1)/2q_4$. Note that  $s+1/2<1$. In this case the inequality  \eqref{e64} yields 
  $\|v\|_{L_{2 q_4'}(\Gamma)}^2\leq \varepsilon \|v\|_{W_{2}^{1}(G)}^{2}+c(\varepsilon) \|v\|_{L_{2}(G)}^{2}$. This inequality and  \eqref{e62} imply that 
 \begin{equation}\label{e65}
 \int_0^t \int_{\Gamma}|(\psi(\xi,x,u_1)-\psi(\xi,x,u_2))v|\,d\Gamma d\xi\leq   \varepsilon \int_0^t \|\nabla v\|_{L_2(G)}^2\,d\xi + c(\varepsilon) t  \|v\|_{L_{\infty}(0,t;L_2(G))}^{2}. 
 \end{equation} 
 Estimate the second summand in  \eqref{e61}. The H\"{o}lder inequality ensures that 
  \begin{equation}\label{e66}
 \int_0^t\int_G|\nabla (u_{1}+u_{2})| |v||\nabla v|\, dxd\xi \leq  \int_0^t\|\nabla (u_{1}+u_{2})\|_{L_q(G)} \|\nabla v\|_{L_2(G)}\| v\|_{L_{2q/(q-2)}(G)}\,d\xi.
 \end{equation} 
 Choose $q=np/(n+2-p)$. In this case  (see Theorem 4.10.2 of Ch. III in \cite{ama1} or Theorem  1.23 in \cite{pya11}) 
 $\|\nabla (u_{1}+u_{2})\|_{L_q(G)}\leq c \|\nabla (u_{1}+u_{2})(\xi,x))\|_{W_p^{1-2/p}(G)}\leq c_1\|u_{1}+u_{2}\|_{W_p^{1,2}(Q)}$.
 As before,  
 $$
 \| v\|_{L_{2q/(q-2)}(G)}\leq  c\|v\|_{W_{2}^{1}(G)}^{2s}\|v\|_{L_{2}(G)}^{2(1-s)}, \  s=n/q<1.
 $$
 In accord with  \eqref{e66}, \eqref{e64}, we infer 
  \begin{multline}\label{e671}
 \int_0^t\int_G|\nabla (u_{1}+u_{2})| |v||\nabla v\|\, dx \leq c_2\int_0^{t}\| \nabla v\|_{L_2(G)}^{1+s}\| v\|_{L_{2q/(q-2)}(G)}^{1-s}\,d\xi\leq \\
 \varepsilon \int_0^t \|\nabla v\|_{L_2(G)}^2\,d\xi + c(\varepsilon) t  \|v\|_{L_{\infty}(0,t;L_2(G))}^{2}, 
  \end{multline} 
 where  $ \varepsilon>0$ is arbitrary. 
 Now take $p>n+2$. 
 The condition  \eqref{e42} implies that 
  \begin{multline}\label{e67}
 \int_0^t\int_G|(b(\xi,x,u_1,\nabla u_1)-b(\xi,x,u_2,\nabla u_2))v|\,dGd\xi\leq \\ \int_0^t\int_G g_1(\xi,x)|v|^2 + g_2(\xi,x) |\nabla v| |v| \,dxd\xi.
 \end{multline} 
  For  the former summand we have the estimate 
\begin{multline}\label{e68}
 \int_0^t\int_G g_1 |v|^2\,d\xi\leq \int_0^t \|g_1\|_{L_{q_6}(G)}  \|v\|_{L_{2q_6'}(G)}^2\,d\xi \leq 
 \int_0^t \|g_1\|_{L_{q_6}(G)}  \|v\|_{W_{2}^1(G)}^{2s}\|v\|_{L_{2}(G)}^{2(1-s)}\,d\xi \leq \\
 \varepsilon \int_0^t \|\nabla v\|_{L_{2}(G)}^{2}\,d\xi + c(\varepsilon)\|v\|_{L_{\infty}(0,t;L_2(G))}^{2}(c(\varepsilon)\int_0^t \|g_1\|_{L_{q_6}(G)}^{1/(1-s)}\,d\xi+\varepsilon t),\ s=n/2q_6,
 \end{multline}
  where we employ the equality  $\|v\|_{W_2^1(G)}^2=\|\nabla v\|_{L_2(G)}^2+\| v\|_{L_2(G)}^2$. 
 Note that  $1/(1-s)\leq q_6$. In view of the absolute continuity of the integral,  $\int_0^t \|g_1\|_{L_{q_6}(G)}^{1/(1-s)}\,d\xi\to 0$ as $t\to 0$. 
 The  integral  $\int_0^t\int_G g_2(\xi,x) |\nabla v| |v| \,dxd\xi$ is estimated similarly. We have
\begin{multline}\label{e69}
\int_0^t\int_G g_2(\xi,x) |\nabla v| |v| \,dxd\xi\leq  
 \int_0^t \|g_2\|_{L_{q_7}(G)}  \|\nabla v\|_{L_{2}(G)}\| v\|_{L_{2q_7'/(2-q_7')}(G)}\,d\xi \leq \\
 \int_0^t  \|g_2\|_{L_{q_7}(G)}  \|v\|_{W_{2}^1(G)}^{1+s}\|v\|_{L_{2}(G)}^{1-s}\,d\xi \leq 
 \varepsilon \int_0^t \|\nabla v\|_{L_{2}(G)}^{2}\,d\xi+ \\ \|v\|_{L_{\infty}(0,t;L_2(G)}^{2}(c(\varepsilon)\int_0^t \|g_2\|_{L_{q_7}(G)}^{2/(1-s)}\,d\xi+\varepsilon t),\ s=n/q_7.
 \end{multline}
Note that  $2/(1-s)\leq q_7$. In this case, the relations   \eqref{e65}, \eqref{e671}, \eqref{e67}, \eqref{e68}, \eqref{e69} and \eqref{e61}, ensure the estimate 
\begin{multline}\label{e70}
 \max_{\tau\in [0,t]}\int\limits_G \frac{v^2(\tau,x)}{2} \,dx+ \int\limits_0^t\int_G\delta_0|\nabla v|^2\,dxd\xi\leq
4\varepsilon \int\limits_0^t \|\nabla v\|_{L_{2}(G)}^{2}\,d\xi+ c(\varepsilon)\|v\|_{L_{\infty}(0,t;L_2(G))}^{2}\varphi(t),
\end{multline}   
where $\varphi(t)$ is a continuous function such that  $\varphi(0)=0$. Choose $\varepsilon$ so that  $4\varepsilon<\delta_0/2$ and $t_0$ so that 
 $ c(\varepsilon)\varphi(t)<1/4$ for $t\leq t_0$. In this case the inequality  $\eqref{e70}$ implies that  $v(t,x)=0$  for $t\leq t_0$. 
Repeating the arguments for  $t\geq t_0$,  we obtain that  $v=0$ on some segment $[t_0,t_1]$, and so on.  It is easy to see that the length of the segments
   $[t_{i-1},t_i]$ does not tend to 0, since it depends on 
  the properties of the functions  $g_i$ and constants arising in the inequalities are really constants from the embedding theorems and interpolaltion inequalities.
  Proceed with the case of  $p\in ((n+2)/2,n+2]$. 
  In this case we have an estimate 
  \begin{multline}\label{e71}
 \int_0^t\int_G|(b(\xi,x,u_1,\nabla u_1)-b(\xi,x,u_2,\nabla u_2))v|\,dGd\xi\leq  \\
 \int_0^t\int_G  g_1 v^2 (1+(|\nabla u_1|+\nabla u_2|)^\alpha) +g_2|\nabla v| |v| (1+ (|\nabla u_1|+\nabla u_2|)^\beta)\,dxd\xi.
 \end{multline} 
The summands  of the form  $\int_0^t\int_G  g_1 v^2 + g_2|\nabla v| |v|\,dxd\xi$ have been already estimated. Estimate the expressions 
$\int_0^t\int_G  g_1 v^2(|\nabla u_1|+\nabla u_2|)^\alpha +g_2|\nabla v| |v| (|\nabla u_1|+\nabla u_2|)^\beta\,dxd\xi$. As in the proof of 
\eqref{e68}, we have 
  \begin{multline}\label{e72}
I_1=\int_0^t\int_G  g_1 v^2(|\nabla u_1|+\nabla u_2|)^\alpha \,dxd\xi\leq 
 \varepsilon \int_0^t \|\nabla v\|_{L_{2}(G)}^{2}\,d\xi+\\  \|v\|_{L_{\infty}(0,t;L_2(G)}^{2}(c(\varepsilon)\int_0^t \|g_1(|\nabla u_1|+|\nabla u_2|)^\alpha\|_{L_{q}(G)}^{1/(1-s)}\,d\xi +\varepsilon t),\ s=n/2q,
\end{multline} 
where $q=q_8np/(np+\alpha q_8(n+2-p))< q_8$ for $p<n+2$. For $p=n+2$, we take $nq_8/2(q_8-1)\geq q<q_8$.  Let  $p<n+2$. The H\"{o}lder inequality 
yields 
  \begin{multline}\label{e73}
\int_0^t \|g_1(|\nabla u_1|+|\nabla u_2)^\alpha|\|_{L_{q}(G)}^{1/(1-s)}\,d\xi \leq \int_0^t \|g_1\|_{L_{q_8}(G)}^{1/(1-s)}
\||\nabla u_1|+|\nabla u_2|\|_{L_{\alpha qq_8/(q_8-q)}(G)}^{\alpha/(1-s)}\,d\xi\leq \\ c \int_0^t \|g_1\|_{L_{q_8}(G)}^{1/(1-s)}\,d\xi,
\end{multline} 
where $\alpha qq_8/(q_8-q)=np/(n+2-p)$ and, thereby, 
$\max_t\||\nabla u_1|+|\nabla u_2|\|_{L_{\alpha qq_8/(q_8-q)}(G)}\leq c(\|u_1\|_{W_p^{1,2}(Q)}+\|u_2\|_{W_p^{1,2}(Q)})$. 
Let  $p=n+2$. Again, we have that 
\begin{multline}\label{e74}
\int_0^t \|g_1(|\nabla u_1|+|\nabla u_2)^\alpha|\|_{L_{q}(G)}^{1/(1-s)}\,d\xi \leq 
 \int_0^t \|g_1\|_{L_{q_8}(G)}^{1/(1-s)}
\||\nabla u_1|+|\nabla u_2|\|_{L_{\alpha qq_8/(q_8-q)}(G)}^{\alpha/(1-s)}\,d\xi\leq\\  c \int_0^t \|g_1\|_{L_{q_8}(G)}^{1/(1-s)}\,d\xi.
\end{multline} 
We have used the fact that,  for $p=n+2$, $\nabla u\in C([0,T];L_q(G))$ for every  $q\geq 1$. 
The condition ov the parameter  $q_8$  in \eqref{e421} ensures the inequality  $1/(1-s)\leq q_8$. 
As before, we infer
  \begin{multline}\label{e75}
 \int_0^t\int_G  g_2 |v| |\nabla v| (|\nabla u_1|+\nabla u_2|)^\beta)\,dxd\xi\leq 
  \varepsilon \int_0^t \|\nabla v\|_{L_{2}(G)}^{2}\,d\xi+ \\ \|v\|_{L_{\infty}(0,t;L_2(G)}^{2}(c(\varepsilon)\int_0^t \|g_2(|\nabla u_1|+|\nabla u_2|)^\beta\|_{L_{q}(G)}^{2/(1-s)}\,d\xi+\varepsilon t),\ s=n/q, 
\end{multline} 
where $q=q_9np/(np+\beta q_9(n+2-p))$ for  $p<n+2$ and $q\in (nq_9/(q_9-1),q_9)$ for $p=n+2$. 
Let us consider the case of  $p<n+2$. We can conclude that 
\begin{multline}\label{e76}
\int_0^t \|g_2(|\nabla u_1|+|\nabla u_2)^\beta|\|_{L_{q}(G)}^{2/(1-s)}\,d\xi \leq \\ \int_0^t \|g_1\|_{L_{q_9}(G)}^{2/(1-s)}
\||\nabla u_1|+|\nabla u_2|\|_{L_{\beta qq_9/(q_9-q)}(G)}^{2\beta/(1-s)}\,d\xi\leq c \int_0^t \|g_1\|_{L_{q_8}(G)}^{2/(1-s)}d\xi,
\end{multline} 
where  $\beta qq_9/(q_9-q)=np/(n+2-p)$ and, thus, the quantity 
$$\max_t\||\nabla u_1|+|\nabla u_2|\|_{L_{2\beta qq_9/(q_9-q)}(G)}\leq  c(\|u_1\|_{W_p^{1,2}(Q)}+\|u_2\|_{W_p^{1,2}(Q)})$$ is finite. 
The same estimate holds for  $p=n+2$. 
The conditions on  $q_9$ in \eqref{e421} ensures the inequality  $2/(1-s)\leq q_9$.
The inequalities  \eqref{e72}-\eqref{e76} imply that the inequality of the form  \eqref{e70} holds which ensures uniqueness of solutions. 
The proofs in the case of the problem  {\rm \eqref{e1}, \eqref{e2}} are much simpler and the sequence of arguments is the same.
So we omit the proofs in this case.   
 \end{proof}

State the results on local solvability of the problems  {\rm \eqref{e1}, \eqref{e2}} and {\rm \eqref{e1}, \eqref{e3}}.   
We present  some new condition on the data. Let $m\leq u_0(x)\leq M$ in $G$ and there exists a constant $m_0>0$ such that
  the functions $a_{ij}(t,x,u)$ are continuous in $\overline{Q}\times (-\infty,\infty)$ and
\begin{equation}\label{e77}
a_{ij}p_ip_j\geq \delta_0|\vec{p}|^2\  \forall \vec{p}=(p_1,\ldots,p_n) \in {\mathbb R}^n, \ (t,x)\in Q, \ -m_0+m\leq u\leq M+m_0,
\end{equation}
where  $\delta_0=const>0$. Let   $R_0=\|\nabla u_0\|_{C(\overline{G})}$. Denote  $B_R=\{(u,\vec{p}):\ |u|+|\vec{p}|\leq R\}$.
The function $b(t,x,u,\vec{p})$ satisfies the Caratheodory condition and
\begin{equation}\label{e78}
b\in L_p(Q;C(B_R))\  \forall R>0. 
\end{equation}
 
 \begin{theorem} \label{thr7}
 Let  the conditions    {\rm \eqref{e2051}, \eqref{e77}, \eqref{e78}} hold  and  $p>n+2$. Then there exists  $\tau_0>0$ such that 
 there exists a solution $u\in W_p^{1,2}(Q_{\tau_0})$  to the problem  {\rm \eqref{e1}, \eqref{e2}}.
       If the conditions {\rm \eqref{e2051}-\eqref{e2053}, \eqref{e77}, \eqref{e78}} hold and  $p>n+2$ then there exists  $\tau_0>0$ such that 
          there exists a solution $u\in W_p^{1,2}(Q_{\tau_0})$ to the problem  {\rm \eqref{e1}, \eqref{e3}}.
If additionally the condition  \eqref{e42} holds then a solution is defined uniquely. 
 \end{theorem}

\begin{proof} Denote   $R=\|\nabla u_0\|_{C(\overline{G})}$. Construct functions  $\varphi,\psi$ such that 
$\varphi(u)=u$ for $u\in [-m_0+m,M+m_0]$,  
$\varphi(u)=M+m_0$ for $u\geq M+m_0$, $\varphi(u)=-m_0+m$ for $u\geq -m_0+m$, $\psi(\vec{p})=\vec{p}$    for $|\vec{p}|\leq 2R$, $\psi(\vec{p})=0$      for  $|\vec{p}|>3R$,
$\psi\in C^1({\mathbb R}^n)$. Define the function  $\tilde{b}(t,x,u,\vec{p})=b(t,x,\varphi(u),\psi(\vec{p}))$. 
Consider the problem 
\eqref{e1}, \eqref{e2} or the problem  \eqref{e1}, \eqref{e3}, where the function  $b(t,x,u,\vec{p}) $ is replaced with  $\tilde{b}(t,x,u,\vec{p})$. 
As is easily seen, the conditions of Theorem  \ref{thr6} are fufilled. 
In this case there exists a solution  $u\in W_p^{1,2}(Q)$ to the problem   \eqref{e1}, \eqref{e2} or, respectively, to the problem \eqref{e1}, \eqref{e3}.
Since  $W_p^{1,2}(Q)\subset C^{1-(n+2)/2p, 2-(n+2)/p}(\overline{Q})$, there exists  $\tau_0>0$ such that  $u(t,x)\in [-m_0+m,M+m_0]$, $\nabla u(t,x)\in B_{2R} $
in $Q_{\tau_0}$. Then  $\varphi(u)=u$, $\psi(\nabla u)=\nabla u$ in  $Q_{\tau_0}$. Hence, the function  $u$ is a solution to the problem 
\eqref{e1}, \eqref{e2} or, respectively,  \eqref{e1}, \eqref{e3} in $Q_{\tau_0}$.
\end{proof}

\section{Discussion}

We examine quasilinear parabolic problems under 
sharp conditions on the data ensuring existence and uniqueness of solutions in Sobolev
classes. The proof relies on estimates resulting from the maximum principle and the H\"{o}lder estimates of the solution. 
 Next, to obtain estimates for higher-order derivatives in $L_p$, we employ the method  of frozen coefficients. 
 The problem is reduced to a problem of finding a fixed point of some compact continuous operator whose solvability is proven  
 by the method of continuation in a parameter.  The local existence theorem is also presented.


\vspace{6pt}




Acknowledgement. This research was supported by the Russian Science Foundation and the Government of the
Khanty-Mansiysk Autonomous Okrug-YUGRA (Grant no. 25-11-20026).

\end{document}